\documentclass[11pt,a4paper]{amsart}
\usepackage{amsmath,amsthm,amsfonts,amscd,amssymb,enumitem,
latexsym,mathrsfs,mathtools,stmaryrd,thmtools,tikz-cd,bm,verbatim}

\usepackage{hyperref}

\usepackage{cleveref}

\newtheorem{theorem}{Theorem}[section]
\newtheorem{corollary}[theorem]{Corollary}
\newtheorem{lemma}[theorem]{Lemma}
\newtheorem{proposition}[theorem]{Proposition}

\newcounter{theoremintro}

\newtheorem{theoremi}[theoremintro]{Theorem}
\newtheorem{corollaryi}[theoremintro]{Corollary}

\theoremstyle{definition}
\newtheorem{definition}[theorem]{Definition}
\newtheorem{remark}[theorem]{Remark}

\newtheorem{example}[theorem]{Example}

\newtheorem*{definition*}{Definition}

\numberwithin{equation}{section}

\allowdisplaybreaks

\newcommand{\mc}[1]{\mathcal{#1}}
\newcommand{\bb}[1]{\mathbb{#1}}

\def\acts{\curvearrowright}

\def\supp{\mathrm{supp}}

\allowdisplaybreaks
\newcommand\blfootnote[1]{%
  \begingroup
  \renewcommand\thefootnote{}\footnote{#1}%
  \addtocounter{footnote}{-1}%
  \endgroup
}
\title{Diagonal dimension and intermediate sub-\texorpdfstring{$\mathrm{C}^\ast$}{C*}-algebras}

\author{Grigoris Kopsacheilis}
\address{Grigoris Kopsacheilis, KU Leuven Department of Mathematics, Celestijnenlaan 200B box 2400, BE-3001 Leuven, Belgium.}
\email{gkopsach@kuleuven.be}
\urladdr{https://sites.google.com/view/gkopsach}

\hypersetup{
  pdftitle={Diagonal Dimension and Intermediate Subalgebras},
  pdfauthor={Grigoris Kopsacheilis}
}

\begin{document}
\begin{abstract}

We show that finiteness of diagonal dimension, in the sense of Li--Liao--Winter, passes to intermediate sub-$\mathrm{C}^\ast$-algebras. More specifically, we obtain an upper bound for the diagonal dimension of the pair of an intermediate sub-C*-algebra in a $\mathrm{C}^\ast$-diagonal together with the diagonal, which we estimate in terms of the diagonal dimension of the ambient pair and the covering dimension of the diagonal's spectrum. This generalizes a theorem of Archbold and Kumjian stating that intermediate sub-$\mathrm{C}^\ast$-algebras of AF-diagonals are AF $\mathrm{C}^\ast$-algebras.
\end{abstract}
\maketitle

\blfootnote{\noindent Funded by the European Union.\ Views and opinions expressed are however those of the author only and do not necessarily reflect those of the European Union or the European Research Council. Neither the EU nor the ERC can be held responsible for them.}

\section*{Introduction}

\noindent Dimension theories are omnipresent across all subjects of mathematics, with various notions of dimension recording different types of data associated to each class of objects. For topological spaces, among the most prominent notions of dimension is that of \emph{covering dimension}: a space has covering dimension at most $d\in\mathbb{N}$ when every finite open cover can be refined to one such that at most $d+1$ of its sets can simultaneously intersect \cite{Engelking78}. With the view of $\mathrm{C}^\ast$-algebras as noncommutative topology via Gelfand duality, the notion of covering dimension was extended to the noncommutative setting with the seminal notions of \emph{nuclear dimension} and \emph{decomposition rank} introduced by Winter and his collaborators \cite{KirWin04, WinZac10} (cf.\ \cite{Win02}). Roughly speaking, just as covering dimension of a space measures its complexity through the least number of overlapping open sets in a refining cover, the nuclear dimension of a $\mathrm{C}^\ast$-algebra measures the least number of finite-dimensional summands needed to approximate the identity map with a sum of completely positive contractive (c.p.c.), orthogonality-preserving (i.e.\ \emph{order zero}) maps -- see \cite[Definition~1.5]{Win18} for the precise definition.

Along with $\mathcal{Z}$-stability (i.e.\ tensorial absorption of the Jiang--Su algebra $\mc{Z}$), nuclear dimension plays a key role in Elliott's classification programme: unital, simple, separable, nuclear, UCT $\mathrm{C}^\ast$-algebras with finite nuclear dimension fall within the scope of the celebrated Classification Theorem \cite[Theorem~5.3]{Whi23}. With such a powerful result at hand, establishing finiteness of the nuclear dimension of various $\mathrm{C}^\ast$-algebras is an imperative task to ensure that they are \emph{classifiable}, i.e.\ they lie in the domain of the Classification Theorem. Some notions of dimension whose theory was developed with one of the main goals being to obtain upper bounds for the nuclear dimension of associated $\mathrm{C}^\ast$-algebras include those of \emph{Rokhlin dimension} \cite{HirshbergWinterZacharias15, Sza15, SimSzaWil20}, (\emph{fine}) \emph{tower dimension} \cite{Ker20}, and  \emph{dynamic asymptotic dimension} \cite{GueWilYu17, Bonicke24}.

After X.\ Li's breakthrough \cite{Li20} on Cartan subalgebras (in the sense of Renault \cite{Ren08}) and $\mathrm{C}^\ast$-diagonals (in the sense of Kumjian \cite{Kum86}; see \Cref{def:diag}) in classifiable $\mathrm{C}^\ast$-algebras, K.\ Li, H.\ Liao and W.\ Winter introduced the \emph{diagonal dimension of sub-$\mathrm{C}^\ast$-algebras} \cite{LiLiaWin23}, a definition akin to that of the nuclear dimension of the ambient $\mathrm{C}^\ast$-algebra that further aims to keep track of the positioning of the sub-$\mathrm{C}^\ast$-algebra -- see \Cref{def:dimdiag}.  In the non-degenerate case, the diagonal dimension of a $\mathrm{C}^\ast$-pair encodes information about the normalizers and the (partial) homeomorphisms that these induce on the spectrum of the abelian sub-$\mathrm{C}^\ast$-algebra: when it is finite, the sub-$\mathrm{C}^\ast$-algebra is a diagonal, and for  diagonals arising from dynamics or groupoids, its value is tightly linked respectively to the fine tower dimension or the dynamic asymptotic dimension of the underlying object (cf.\ \cite[Theorems~B, E]{LiLiaWin23}), thus making it a prominent definition for a potential (classification and) regularity theory for $\mathrm{C}^\ast$-diagonal pairs (cf.\ \cite[Problem~XLVIII]{SchTikWhi25}).

Nuclear dimension enjoys many permanence properties (e.g.\ it is non-increasing with respect to hereditary sub-$\mathrm{C}^\ast$-algebras), but it is not monotone with respect to arbitrary sub-$\mathrm{C}^\ast$-algebras. This is rather natural since a $\mathrm{C}^\ast$-algebra can contain much wilder sub-$\mathrm{C}^\ast$-algebras, with Kirchberg's $\mc{O}_2$-embedding and $M_{2^\infty}$-subquotient theorems \cite[Theorem~IV.3.4.18]{Bla06} being prime examples of this phenomenon -- in particular, an arbitrary sub-$\mathrm{C}^\ast$-algebra of a nuclear $\mathrm{C}^\ast$-algebra need not be nuclear. Nevertheless, Archbold and Kumjian showed in \cite{ArcKum86} that some well-positioned sub-$\mathrm{C}^\ast$-algebras retain nuclearity: more precisely, they proved that sub-$\mathrm{C}^\ast$-algebras of uniformly hyperfinite (UHF) algebras that contain an \emph{approximately finite-dimensional} (\emph{AF}) \emph{diagonal}, are not only nuclear, but even AF $\mathrm{C}^\ast$-algebras (see also \cite{Archbold99}). Recall that a diagonal $(D\subset A)$ is an \emph{AF-diagonal} when it is the inductive limit of a system of masas $(D_{F_i}\subset F_i)$ in finite-dimensional $\mathrm{C}^\ast$-algebras with normalizer-preserving connecting maps. Every AF algebra admits an AF-diagonal \cite{StrVoi75}, and conceptually understanding when a sub-$\mathrm{C}^\ast$-algebra is an AF-diagonal is a problem of current interest \cite[Problem~XLVII]{SchTikWhi25}.

 Just as nuclear dimension zero is equivalent to the $\mathrm{C}^\ast$-algebra being an AF algebra, diagonal dimension zero of a (non-degenerate) sub-$\mathrm{C}^\ast$-algebra is equivalent to the pair being an AF-diagonal \cite[Theorem~4.1]{LiLiaWin23}. This allows us to interpret the theorem of Archbold and Kumjian as a statement for dimension zero, which brings us to the main result of this paper. In what follows we use the notation $\dim^{+1}(.)$ as shorthand for $\dim(.)+1$.

\begin{theoremi}\label{thmi:A}
Let $(D\subset A)$ be a $\mathrm{C}^\ast$-diagonal with $D$ separable. If $D\subset B\subset A$ is an intermediate sub-$\mathrm{C}^\ast$-algebra, then
\[
\dim_{\mathrm{diag}}^{+1}(D\subset B)\le \dim^{+1}(\widehat{D})\cdot\dim^{+1}_\mathrm{diag}(D\subset A),
\]
where $\widehat{D}$ is the Gelfand spectrum of $D$. In particular, if $D$ has totally disconnected spectrum, then 
\[
\dim_\mathrm{diag}(D\subset B)\le \dim_\mathrm{diag}(D\subset A).
\]
\end{theoremi}

When $\dim_{\mathrm{diag}}(D\subset A)=0$, and thus also $\dim\widehat{D}=0$, \Cref{thmi:A} recovers (a strengthening of) the result of Archbold and Kumjian. Note that the dimension can drop when considering an intermediate sub-$\mathrm{C}^\ast$-algebra: for example, the $\mathrm{C}^\ast$-pair $(C(X)\subset C(X)\rtimes\bb{Z})$ has diagonal dimension $1$ for any Cantor minimal system $\bb{Z}\acts X$, whereas the pair $(C(X)\subset C(X))$ has dimension $0$. A somewhat more interesting example also occurs by considering Putnam's \emph{orbit breaking subalgebra} $C(X)\subset A_{\{y\}}\subset C(X)\rtimes\bb{Z}$ for some $y\in X$ (cf.\ \cite{Putnam89, DeeleyPutnamStrung24}). The pair $(C(X)\subset A_{\{y\}})$ also has dimension $0$. Beyond totally disconnected spectra, it remains unclear whether the diagonal dimension can increase when passing to an intermediate sub-$\mathrm{C}^\ast$-algebra -- we discuss this in \Cref{rmk:optimal}. 

The following corollary of \Cref{thmi:A} shows that many intermediate sub-$\mathrm{C}^\ast$-algebras of pairs with finite diagonal dimension fall within the scope of the Classification Theorem.

\begin{corollaryi}\label{cori:B}
Let $(D\subset A)$ be a unital $\mathrm{C}^\ast$-diagonal with $A$ separable. If $\dim_\mathrm{diag}(D\subset A)<\infty $, then any  $\mathrm{C}^\ast$-algebra $B$ with $D\subset B\subset A$ has finite nuclear dimension. If moreover $B$ is simple, it is classifiable.
\end{corollaryi}

Our proof of \Cref{thmi:A} is based on the strategy of Archbold--Kumjian developed in \cite{ArcKum86}. In their context, starting from an AF-diagonal in a UHF algebra -- say $(D\subset A)=\varinjlim_k(\bigotimes_{i=1}^kD_{r_i}\subset \bigotimes_{i=1}^kM_{r_i})$ -- and an intermediate sub-$\mathrm{C}^\ast$-algebra $B$, they first obtain a sequence of conditional expectations $\Psi_k\colon A\to C_k\coloneqq \bigotimes_{i=1}^kM_{r_i}\otimes \bigotimes_{i>k}D_{r_i}$ with $\Psi_k(B) = B\cap C_k$ and $\Psi_k\to\mathrm{id}_A$ in the point-norm topology. Second, they prove that for a totally disconnected compact space $X$ and $r\in\bb{N}$, any sub-$\mathrm{C}^\ast$-algebra $D_r\otimes C(X) \subset B_0\subset M_r\otimes C(X)$ is AF, and thus each $\Psi_k(B)$ is AF, so $B=\overline{\bigcup_k\Psi_k(B)}$ is AF. 

In our generality, the first obstacle is that we lack an inductive limit structure for a diagonal $(D\subset A)$ for which we only know its dimension is at most $d\in\bb{N}$; nevertheless we do have $d+1$ normalizer-preserving c.p.c.\ order zero maps $\varphi^{(i)}\colon F^{(i)}\to A$ in our approximation of the identity, and so the role of the auxiliary algebras $\bigotimes_{i=1}^kM_{r_i}\otimes\bigotimes_{i>k}D_{r_i}$ is played simultaneously by the $d+1$ $\mathrm{C}^\ast$-algebras $\mathrm{C}^\ast(D,\varphi^{(i)}(F^{(i)}))$. For intermediate sub-$\mathrm{C}^\ast$-algebras therein, we are able to find (sufficiently) \emph{approximate} conditional expectations to play the role of the $\Psi_k$'s -- see \Cref{prop:cond-exps}. It is crucial to ensure that our replacements of the $\Psi_k$'s map intermediate subalgebras within themselves (which was automatic for AF-diagonals in UHF algebras), which we achieve using a result of Donsig--Pitts \cite[Proposition~3.10]{DonsigPitts08}. The next difficulty lies in the approximations needed. Similar to the case of an intermediate $D_r\otimes C(X)\subset B_0\subset M_r\otimes C(X)$, we show that (a large part of) an intermediate subalgebra $D\subset B\subset \mathrm{C}^\ast(D,\varphi^{(i)}(F^{(i)}))$ has, in a sense, a \emph{coordinate system} that we can employ (\Cref{cor:ideal-subhom}). At this point it becomes clear that $(D\subset B)$ is also a diagonal, however calculating the exact value of its dimension requires some more work. The following result is the analogue of the second step in the approach of Archbold--Kumjian.

\begin{theoremi}\label{thmi:C}
Let $(D\subset A)$ be a $\mathrm{C}^\ast$-diagonal with $D$ separable, let $(D_F\subset F)$ be a masa in a finite-dimensional $\mathrm{C}^\ast$-algebra, and let  $\varphi\colon F\to A$ be a linear c.p.c.\ order zero map with $\varphi(\mc{N}_F(D_F))\subset \mc{N}_A(D)$. Then, any $\mathrm{C}^\ast$-algebra $B$ with $D\subset B\subset\mathrm{C}^\ast(D,\varphi(F))$ is subhomogeneous and 
\[
\dim_\mathrm{diag}(D\subset B)=\dim\widehat{D}.
\]
\end{theoremi}

The above calculation is also the reason why the factor $\dim^{+1}\widehat{D}$ appears in the estimate in \Cref{thmi:A}. There is some technical effort needed for the proof of \Cref{thmi:C} which we now explain, and to isolate the core ideas let us assume that $D$ is unital and $F=M_r$ is a matrix algebra. Assume for a second that $\mathrm{C}^\ast(D,\varphi(M_r))\cong M_r\otimes C(X)$ for some $X\subset \widehat{D}$ with $D\cong D_r\otimes C(X)$. An intermediate subalgebra $D_r\otimes C(X) \subset B\subset M_r\otimes C(X)$ is described as $r\times r$ matrices over $C(X)$ for which there are open subsets $V_{ij}$ of $X$ such that only functions supported within $V_{ij}$ are allowed to appear as the $(i,j)$ entry of a matrix contained in $B$ (cf.\ \Cref{lem:ideal-matrix-subhom}). When $B= M_r\otimes C(X)$ all the sets $V_{ij}$ are equal to $X$, and in order to show $\dim_{\mathrm{diag}}(D\subset B)$ is equal to $\dim\widehat{D}=\dim X\eqqcolon d$ one can simply use a good enough $(d+1)$-colored partition of unity $\{h_s: s\in \bigsqcup_{i=0}^dS_i\}$ for $X$ to c.p.c.\ order zero maps from each $\bb{C}^{S_i}$ to $C(X)$, and amplify those with $M_r$. However one cannot merely use that argument for an arbitrary intermediate $B$, since the ranges of these maps can easily fail to lie within $B$, which is crucial. To remedy this, we replace the ``too large" amplifying algebras $M_r$ as follows: for each $s\in S_i$, instead of adding a full copy of $M_r$, we use a direct sum of corners of $M_r$ with respect to diagonal projections corresponding to the equivalence classes of the relation $i\sim j\iff\supp(h_s)\subset V_{ij}$. In this way, we ensure that the ranges in our approximation never leave the allowed coordinates that define $B$.

The final problem is that in our generality the order zero map $\varphi$ is not necessarily unital. In the unital case $\varphi$ is automatically a $^*$-homomorphism and we are essentially in the situation described in the above paragraph, so the spectrum of $D$ decomposes as the disjoint union of $r$ disjoint clopen copies of $X$. Without this assumption, we only have $r$ disjoint open copies of $X$ in $\widehat{D}$ with their union having nonempty complement in $\widehat{D}$, and so we cannot simply start with a suitable $(d+1)$-colored partition of unity for $X$ and replicate it $r$ times across its open copies, as this process does not produce a ($(d+1)$-colored) partition of unity for $\widehat{D}$; we thus need to find a good enough $(d+1)$-colored partition of unity \emph{for the entire space $\widehat{D}$}, that (in a sense) behaves \emph{in the same way} along the $r$ disjoint copies of $X$, and use this to define our order zero maps. We handle this in \Cref{prop:POU}.

Intermediate sub-$\mathrm{C}^\ast$-algebras of $\mathrm{C}^\ast$-diagonals were also studied in \cite{BrownExelFullerPittsReznikoff21, BrownExelFullerPittsReznikoff21corrigendum} (cf.\ \cite{Komura22}), where it was shown that in a nuclear $\mathrm{C}^\ast$-algebra, pairing the diagonal with an intermediate sub-$\mathrm{C}^\ast$-algebra gives again a $\mathrm{C}^\ast$-diagonal. Also, the theorem of Archbold--Kumjian was recently applied in \cite{KopWin26} to give the first examples of Cantor spectrum diagonals in $M_{2^\infty}$ that are not AF-diagonals. It is possible that \Cref{thmi:A} can also serve in more general contexts as a means of establishing non-conjugacy of diagonals with homeomorphic spectra in various $\mathrm{C}^\ast$-algebras by studying their intermediate sub-$\mathrm{C}^\ast$-algebras.

\subsection*{Acknowledgments}
I am grateful to Wilhelm Winter for fruitful discussions about the topic that were vital to this project.

The author has been supported by the European Research Council under the European Union's Horizon Europe research and innovation programme (ERC grant AMEN-101124789).

\section{Preliminaries}
\noindent For a $\mathrm{C}^\ast$-algebra $A$, we write $A^1$ for its closed unit ball and $A_+$ for the cone of positive elements. We write $A_+^1\coloneqq A_+\cap A^1$. For $a,b\in A$ we say that $a,b$ are \emph{orthogonal} when $ab=ba=0$ and write $a\perp b$. A completely positive (c.p.) linear map between $\mathrm{C}^\ast$-algebras is said to be \emph{order zero} when it preserves orthogonality. When $\mc{F}$ is a finite subset of $A$, we write $\mc{F}\Subset A$. The closed, two-sided ideal in $A$ generated by a set $S\subset A$ is denoted by $\langle S\rangle\trianglelefteq A$. A sub-$\mathrm{C}^\ast$-algebra $B\subset A$ is called \emph{non-degenerate} when $B$ contains an approximate unit for $A$. $\mc{M}(A)$ denotes the multiplier algebra of $A$.

\begin{definition}{(cf.\ \cite{Kum86}).}\label{def:diag}
A \emph{$\mathrm{C}^\ast$-diagonal} is a sub-$\mathrm{C}^\ast$-algebra $(D\subset A)$ such that
\begin{enumerate}[label=(\roman*)]
\item\label{it:superf} $(D\subset A)$ is non-degenerate,
\item  $D$ is a maximal abelian self-adjoint subalgebra (\emph{masa}) of $A$ with the \emph{unique extension property}, i.e.\ every pure state on $D$ extends uniquely to a (pure) state on $A$,
\item $D$ is \emph{regular} in $A$, i.e.\ the set of \emph{normalizers} 
\[
\mc{N}_A(D)\coloneqq \{v\in A: vDv^*\cup v^*Dv\subset D\}
\] 
generates $A$ as a $\mathrm{C}^\ast$-algebra,
\item $D$ is the image of a faithful conditional expectation $A\to D$.
\end{enumerate}
\end{definition}
Condition \ref{it:superf} in the preceding definition is in fact superfluous \cite{Pit21}. It is a consequence of the spectral theorem that a masa $(D_F\subset F)$ in a finite-dimensional $\mathrm{C}^\ast$-algebra is automatically a $\mathrm{C}^\ast$-diagonal, and any two masas in $F$ are conjugate, that is, there is an automorphism of $F$ mapping one masa onto the other \cite[Example~1.5]{LiLiaWin23}.

For $r\in\bb{N}$ we write $(D_r\subset M_r)$ for the $\mathrm{C}^\ast$-diagonal of $r\times r$ diagonal matrices in $r\times r$ matrices. For a finite-dimensional $\mathrm{C}^\ast$-algebra $F=\bigoplus_{n=1}^NM_{r_n}$, since any masa $D_F\subset F$ is conjugate to $\bigoplus_{n=1}^ND_{r_n}$, we have that $F$ admits a system of matrix units \emph{with respect to $D_F$}, i.e.\ a set of elements $\{e_{ij}^{(n)}:1\le i,j\le r_n, 1\le n\le N\}$ that satisfies the usual matrix-unit multiplication and adjoint relations, that spans $F$, and the pairwise-orthogonal projections $\{e_{ii}^{(n)}: 1\le i\le r_n, 1\le n\le N\}$ generate $D_F$.

\begin{definition}\label{def:dimdiag}(cf.\ \cite[Definition~A]{LiLiaWin23}).
Let $(D\subset A)$ be a sub-$\mathrm{C}^\ast$-algebra and let $d\in\bb{N}$. We say that $(D\subset A)$ has \emph{diagonal dimension at most} $d$ and write $\dim_\mathrm{diag}(D\subset A)\le d$ when, for any $\varepsilon>0$ and $\mc{F}\Subset A$ there is a finite-dimensional $\mathrm{C}^\ast$-algebra $F=\bigoplus_{i=0}^dF^{(i)}$ together with a masa $D_F=\bigoplus_{i=0}^dD_{F^{(i)}}$ and linear c.p.\ maps $\psi\colon A\to F$ and $\varphi\colon F\to A$ such that
\begin{enumerate}[label=(\roman*)]
\item $\|\varphi\psi(a)-a\|<\varepsilon$ for all $a\in\mc{F}$,
\item $\psi$ is completely positive contractive (c.p.c.),
\item $\psi(D)\subset D_F$,
\item for each $i=0,\dots,d$ the map $\varphi^{(i)}\coloneqq\varphi\vert_{F^{(i)}}$ is c.p.c.\ order zero, and \emph{normalizer-preserving}, i.e.\ $\varphi^{(i)}(\mc{N}_{F^{(i)}}(D_{F^{(i)}}))\subset \mc{N}_A(D)$.
\end{enumerate}
The \emph{diagonal dimension} of $(D\subset A)$ is the least $d\in\bb{N}\cup\{\infty\}$ for which $\dim_\mathrm{diag}(D\subset A)\le d$, and we denote it by $\dim_\mathrm{diag}(D\subset A)$.
\end{definition}

Recall that if $\varphi\colon A\to B$ is a positive linear map between $\mathrm{C}^\ast$-algebras and $(D_A\subset A)$, $(D_B\subset B)$ are abelian sub-$\mathrm{C}^\ast$-algebras such that $\varphi(\mc{N}_A(D_A))\subset \mc{N}_B(D_B)$, then $\varphi(D_A)\subset D_B$ \cite[Proposition~1.9]{LiLiaWin23}.

For a locally compact Hausdorff space $X$, we write $\dim X$ for its covering dimension \cite[Definition~1.6.7]{Engelking78}. For $f\in C_0(X)$, the \emph{support} of $f$ is the open set $\supp(f)\coloneqq f^{-1}(\bb{C}\setminus\{0\})$, and we refer to its closure as the \emph{closed support} of $f$. We write $C_c(X)$ for the (not necessarily closed) two-sided ideal of those $f\in C_0(X)$ whose closed support is compact. When $U\subset X$ is open, we view $C_0(U)$ as an ideal in $C_0(X)$ by identifying it with those $f\in C_0(X)$ such that $\supp(f)\subset U$.

For $d\in\bb{N}$, we say that a family of functions $\{h_s\}_{s\in S}\subset C_0(X)$ is $(d+1)$\emph{-colorable} when there is a map $c\colon S\to \{0,\dots,d\}$ such that, if $s_1,s_2\in S$ are distinct with $c(s_1)=c(s_2)$, then $h_{s_1}\perp h_{s_2}$. We call such a map a \emph{coloring map} for $\{h_s\}_{s\in S}$. For $K\subset X$, we say that $\{h_s\}_{s\in S}\subset C_0(X)_+^1$ is a \emph{partition of unity for} $K$ when $S$ is finite and $\sum_{s\in S}h_s(x)=1$ for all $x\in K$. Recall that when $K$ is compact, for any finite open cover $\mc{U}$ of $K$ there is a partition of unity $\{h_U\}_{U\in\mc{U}}\subset C_c(X)_+^1$ for $K$ \emph{subordinate to} $\mc{U}$, meaning that the closed support of each $h_U$ is contained in $U$. In particular, by Ostrand's theorem \cite[Theorem~3.2.4]{Engelking78} when $X$ is a locally compact metric space with $\dim X\le d$, for any $\delta>0$ and any compact $K\subset X$ there is a $(d+1)$-colorable partition of unity $\{h_s\}_{s\in S}\subset C_c(X)_+^1$ such that the closed support of each $h_s$ has diameter less than $\delta$.

\section{Approximate conditional expectations preserving intermediate subalgebras}

\noindent In this section, starting from a $\mathrm{C}^\ast$-diagonal $(D\subset A)$ and a linear c.p.c.\ order zero map $\varphi\colon F\to A$ with $F$ a finite-dimensional $\mathrm{C}^\ast$-algebra such that there is a masa $(D_F\subset F)$ for which $\varphi$ is normalizer-preserving, we will show that there exists an ``approximate" conditional expectation $A\to\mathrm{C}^\ast(D,\varphi(F))$ that preserves intermediate subalgebras (cf.\ \Cref{prop:cond-exps}).

\begin{lemma}\label{lem:orthogonal-range-support-0exp}
Let $(D\subset A)$ be an abelian sub-$\mathrm{C}^\ast$-algebra and let $\Phi\colon A\to D$ be a conditional expectation. If $v\in A$ is such that $vv^*, v^*v\in D$ and $vv^*\perp v^*v$, then $\Phi(v)=0$.
\end{lemma}

\begin{proof}
Note that $v=\lim_nv(v^*v)^\frac{1}{n}=\lim_n (vv^*)^\frac{1}{n}v$. For each $n\ge1$ let $\{p_{k,n}\}_{k\in\bb{N}}$ be a sequence of polynomials on $\bb{C}$ with zero constant term that converge uniformly over $\sigma(v^*v)\cup\{0\}=\sigma(vv^*)\cup\{0\}$ to the map $z\mapsto z^\frac{1}{n}$ (which exists by Weierstra\ss 's approximation theorem). Using that $D$ lies in the multiplicative domain of $\Phi$, we have that
\begin{align*}
\Phi(v)&=\lim_n\Phi(v(v^*v)^\frac{1}{n})\\
&=\lim_n\Phi(v)\cdot (v^*v)^\frac{1}{n}\\
&=\lim_n\lim_k\Phi(v)\cdot p_{k,n}(v^*v)
\end{align*}
so it suffices to show that $\Phi(v)v^*v=0$. But indeed,
\begin{align*}
\Phi(v)v^*v&=\lim_n\Phi((vv^*)^\frac{1}{n}v)v^*v\\
&=\lim_n(vv^*)^\frac{1}{n}\Phi(v)v^*v\\
&=\lim_n(vv^*)^\frac{1}{n}\cdot v^*v\cdot \Phi(v)\\
&=0
\end{align*}
where in the last step we used that $v^*v\perp vv^*$.
\end{proof}

\begin{lemma}\label{lem:dense-subspace}
Let $(D\subset A)$ be a non-degenerate sub-$\mathrm{C}^\ast$-algebra with $D$ abelian, let $(D_F\subset F)$ be a masa in a finite-dimensional $\mathrm{C}^\ast$-algebra, and let $\varphi\colon F\to A$ be a c.p.c.\ order zero map such that $\varphi(\mc{N}_F(D_F))\subset\mc{N}_A(D)$. Let $\{e_{i,j}^{(n)}\}_{i,j,n}$ be a system of matrix units of $F$ relative to $D_F$. Then
\[
D+\sum_{i,j,n}D\cdot \varphi(e_{ij}^{(n)})
\]
is a dense linear subspace of $\mathrm{C}^\ast(D,\varphi(F))$.
\end{lemma}

\begin{proof}
Set $S\coloneqq D+\sum_{i,j,n}D\cdot \varphi(e_{ij}^{(n)})$. Clearly $S\subset \mathrm{C}^\ast(D,\varphi(F))$. Since $S$ contains $D$ and since $D$ contains an approximate unit for $A$, by the definition of $S$ we have that $\varphi(F)\subset \overline{S}$, so it suffices to show that $\overline{S}$ is closed under adjoints and multiplication. Arguing as in the beginning of the proof of \Cref{lem:orthogonal-range-support-0exp}, we see that if $v\in\mc{N}_A(D)$, then $vD\subset\overline{Dv}$ and $Dv\subset\overline{vD}$. Using this observation together with the fact that each $\varphi(e_{i,j}^{(n)})$ is a normalizer of $D$, it follows that $S\cdot S\subset \overline{S}$ and that $S^*\subset\overline{S}$, and thus $\overline{S}$ is a $\mathrm{C}^\ast$-algebra.
\end{proof}

\begin{proposition}\label{prop:cond-exps}
Let $(D\subset A)$ be a $\mathrm{C}^\ast$-diagonal, let $(D_F\subset F)$ be a masa in a finite-dimensional $\mathrm{C}^\ast$-algebra, and let $\varphi\colon F\to A$ be a c.p.c.\ order zero map with $\varphi(\mc{N}_F(D_F))\subset\mc{N}_A(D)$. 

Then, for any $\mathcal{F}\Subset\mathrm{C}^\ast(D,\varphi(F))$ and any $\eta>0$ there is a faithful c.p.c.\ linear map $\Psi\colon A\to \mathrm{C}^\ast(D,\varphi(F))$ such that 
\begin{enumerate}[label=\normalfont{(\roman*)}]
\item\label{item:condexp1} $\Psi(B)\subset B$ for all $\mathrm{C}^\ast$-algebras $D\subset B\subset A$,
\item\label{item:condexp2} $\Psi\vert_D=\mathrm{id}_D$, and
\item\label{item:condexp3} $\|\Psi(a)-a\|<\eta$ for all $a\in\mathcal{F}$.
\end{enumerate}

\end{proposition}

\begin{proof}
Assume first that $A$ is a unital $\mathrm{C}^\ast$-algebra (and thus $1_A\in D$). Let $\{e_{ij}^{(n)}: 1\le i,j\le r_n, 1\le n\le N\}$ be a system of matrix units for $F$ relative to $D_F$, let $\Phi\colon A\to D$ denote the (unique) conditional expectation of the $\mathrm{C}^\ast$-diagonal pair $(D\subset A)$, let $\mathcal{F}\Subset\mathrm{C}^\ast(D,\varphi(F))$ and let $\eta>0$. By \Cref{lem:dense-subspace}, for each $a\in\mc{F}$ we can find $d_a\in D$ and $a_{i,j}^{(n)}\in D$ such that with $b_a:=d_a+\sum_{i,j,n}a_{i,j}^{(n)}\varphi(e_{ij}^{(n)})$ we have that $\|b_a -a \|<\eta/3$ for all $a\in\mc{F}$. Let $M\coloneqq \max_{a\in\mc{F}}\sum_{n=1}^N\sum_{i,j=1}^{r_n}\|a_{i,j}^{(n)}\|$ and set $\varepsilon\coloneqq (3M+1)^{-1}\eta>0$. If $\Psi\colon A\to\mathrm{C}^\ast(D,\varphi(F))$ is a c.p.c.\ linear map with $\Psi\vert_D=\mathrm{id}_D$ and such that $\|\Psi(\varphi(e_{ij}^{(n)}))-\varphi(e_{ij}^{(n)})\|<\varepsilon$ for all $i,j,n$, then using that $D$ lies in the multiplicative domain of $\Psi$ and applying the triangle inequality yields that $\|\Psi(a)-a\|<\eta$ for all $a\in\mc{F}$.

Let $\beta>0$ be so small so that $\sup_{t\in[0,1]}|t-t^{1+2\beta}|<\varepsilon$. Set $h\coloneqq 1_A-\varphi(1_F)^{2\beta}\in D_+^1$ and note that $h$ commutes with $\mathrm{C}^\ast(D,\varphi(F))$. Using order zero functional calculus \cite[Corollary~3.2]{WinZac09} we define the map $\Psi\colon A\to \mathrm{C}^\ast(D,\varphi(F))$ as
\begin{equation}\label{eq:cond-exp-def}
\Psi(a)\coloneqq h \cdot \Phi(a) + \sum_{n=1}^N\sum_{i,j=1}^{r_n}\Phi\big(a\cdot \varphi^\beta(e_{ji}^{(n)})\big) \cdot\varphi^\beta(e_{ij}^{(n)}),\quad a\in A.
\end{equation}
We first show that $\Psi$ is completely positive. It suffices to show that the map 
\begin{equation}
\Psi_0(a)\coloneqq\sum_{n=1}^N\sum_{i,j=1}^{r_n}\Phi\big(a\cdot \varphi^\beta(e_{ji}^{(n)})\big) \cdot\varphi^\beta(e_{ij}^{(n)}),\quad a\in A
\end{equation}
is completely positive. We first observe that for all $a\in A$, $1\le n\le N$ and $1\le i,j\le r_n$ we have that
\begin{equation}\label{eq:cp-formula}
\Phi\big(a\cdot \varphi^\beta(e_{ji}^{(n)})\big)\cdot \varphi^\beta(e_{ij}^{(n)}) = \varphi^\frac{\beta}{2}(e_{i1}^{(n)})\cdot \Phi\big(\varphi^\frac{\beta}{2}(e_{1i}^{(n)})\cdot a \cdot \varphi^\frac{\beta}{2}(e_{j1}^{(n)})\big) \cdot \varphi^\frac{\beta}{2}(e_{1j}^{(n)}).
\end{equation}
Before proving this we recall three facts that will be necessary for our calculations. First, we have $\varphi^{\alpha_1+\alpha_2}(x\cdot y)=\varphi^{\alpha_1}(x) \cdot \varphi^{\alpha_2}(y)$ for all $\alpha_1,\alpha_2>0$ and all $x,y\in F$. Second, the functional calculi of normalizer-preserving order zero maps are again normalizer-preserving by \cite[Remark~4.2]{KopWin24}, and third, if $v\in\mc{N}_A(D)$, we have that $v\Phi(.)v^*=\Phi(v\; . \; v^*)$ by \cite[Lemma~3.2]{CryNag17}. Now starting to calculate from the right side of \eqref{eq:cp-formula} we have that
\begin{align*}
&\varphi^\frac{\beta}{2}(e_{i1}^{(n)}) \cdot  \Phi\big(\varphi^\frac{\beta}{2}(e_{1i}^{(n)}) \cdot  a \cdot  \varphi^\frac{\beta}{2}(e_{j1}^{(n)})\big) \cdot  \varphi^\frac{\beta}{2}(e_{1j}^{(n)}) =\\
& \varphi^\frac{\beta}{2}(e_{i1}^{(n)}) \cdot \Phi\big(\varphi^\frac{\beta}{4}(e_{1i}^{(n)}) \cdot \varphi^\frac{\beta}{4}(e_{ii}^{(n)})\cdot a\cdot \varphi^\frac{\beta}{4}(e_{ji}^{(n)})\cdot \varphi^\frac{\beta}{4}(e_{i1}^{(n)})\big) \cdot \varphi^\frac{\beta}{2}(e_{1j}^{(n)}) =\\
& \varphi^\frac{\beta}{2}(e_{i1}^{(n)}) \cdot \varphi^\frac{\beta}{4}(e_{1i}^{(n)}) \cdot \varphi^\frac{\beta}{4}(e_{ii}^{(n)}) \cdot \Phi\big(a \cdot \varphi^\frac{\beta}{4}(e_{ji}^{(n)})\big)\cdot \varphi^\frac{\beta}{4}(e_{i1}^{(n)})\cdot \varphi^\frac{\beta}{2}(e_{1j}^{(n)})=\\
&\varphi^\beta(e_{ii}^{(n)})\cdot \Phi\big(a \cdot \varphi^\frac{\beta}{4}(e_{ji}^{(n)})\big) \cdot \varphi^\frac{3\beta}{4}(e_{ij}^{(n)})=\\
&\Phi\big(a\cdot \varphi^\frac{\beta}{4}(e_{ji}^{(n)})\big) \cdot \varphi^\beta(e_{ii}^{(n)}) \cdot \varphi^\frac{3\beta}{4}(e_{ij}^{(n)})=\\
&\Phi\big(a \cdot \varphi^\frac{\beta}{4}(e_{ji}^{(n)})\cdot \varphi^\frac{3\beta}{4}(e_{ii}^{(n)})\big) \cdot \varphi^\frac{\beta}{4}(e_{ii}^{(n)}) \cdot \varphi^\frac{3\beta}{4}(e_{ij}^{(n)})=\\
&\Phi\big(a \cdot \varphi^\beta(e_{ji}^{(n)})\big)\cdot \varphi^\beta(e_{ij}^{(n)}),
\end{align*}
as we wanted. For $1\le n\le N$, set
\begin{equation}
v_n\coloneqq\begin{bmatrix}\varphi^\frac{\beta}{2}(e_{11}^{(n)})&\varphi^\frac{\beta}{2}(e_{21}^{(n)})&\dots & \varphi^\frac{\beta}{2}(e_{r_n1}^{(n)}) \end{bmatrix}\in M_{1\times r_n}(A)
\end{equation}
and let $\theta_n^{(1)}\colon A\to M_{r_n}(A)$ and $\theta_n^{(2)}\colon M_{r_n}(A)\to A$ be the maps defined by $\theta_n^{(1)}(.)\coloneqq v_n^*\; .\; v_n$ and $\theta_n^{(2)}(.)\coloneqq v_n \; . \; v_n^*$, which are completely positive. Using \eqref{eq:cp-formula}, we have that
\begin{equation}
\Psi_0 = \sum_{n=1}^N\theta_n^{(2)}\circ (\Phi\otimes\mathrm{id}_{M_{r_n}})\circ \theta_n^{(1)}
\end{equation}
which is completely positive as a sum of completely positive maps.

To see that $\Psi$ is faithful, let $a\in A$ be such that $\Psi(a^*a)=0$. Then $\Phi(a^*a)h=0$ and $\sum_{n=1}^N\sum_{i,j=1}^{r_n}\Phi\big(a^*a\cdot \varphi^\beta(e_{ji}^{(n)})\big)\cdot \varphi^\beta(e_{ij}^{(n)})=0$. Applying $\Phi$ to the latter equation, by \Cref{lem:orthogonal-range-support-0exp} and the fact that $D$ lies in the multiplicative domain of $\Phi$, we obtain
\begin{equation}\label{eq:psi-faithful-second}
\sum_{n=1}^N\sum_{j=1}^{r_n}\Phi(a^*a)\varphi^{2\beta}(e_{jj}^{(n)})=0.
\end{equation}
Now \eqref{eq:psi-faithful-second} can be re-written as $\Phi(a^*a)\varphi(1_F)^{2\beta}=0$, and so combining this with $\Phi(a^*a)h=0$ we get $\Phi(a^*a)=0$. Since $\Phi$ is faithful, we get $a=0$ as we wanted.

Moreover, if $D\subset B\subset A$ is an intermediate sub-$\mathrm{C}^\ast$-algebra, by \cite[Proposition~3.10]{DonsigPitts08} we have that $\Phi(bv)v^*\in B$ for all $b\in B$ and all $v\in\mathcal{N}_A(D)$. Since order zero calculus of a normalizer-preserving order zero map is also a normalizer-preserving order-zero map, combining these with a direct inspection of \eqref{eq:cond-exp-def} we get that $\Psi(B)\subset B$, which establishes \ref{item:condexp1}.

For $f\in D$, using in the first and second line below respectively that $D$ is in the multiplicative domain of $\Phi$ and \Cref{lem:orthogonal-range-support-0exp}, we have that
\begin{align*}
\Psi(f) &= h f + \sum_{n=1}^N\sum_{i,j=1}^{r_n}f\cdot \Phi(\varphi^\beta(e_{ji}^{(n)}))\cdot\varphi^\beta(e_{ij}^{(n)})\\
&=  hf + \sum_{n=1}^N\sum_{j=1}^{r_n}f\cdot\varphi^\beta(e_{jj}^{(n)})\cdot\varphi^\beta(e_{jj}^{(n)})\\
&=hf + \sum_{n=1}^N\sum_{j=1}^{r_n}f\cdot \varphi^{2\beta}(e_{jj}^{(n)})\\
&= hf + \varphi^{2\beta}(1_F)f\\
&= (h+\varphi(1_F)^{2\beta})f\\
&=f
\end{align*}
and so $\Psi$ is the identity on $D$ (which also contains $\varphi(e_{jj}^{(n)})$ for all $1\le j\le r_n$ and all $1\le n\le N$). This establishes \ref{item:condexp2} and also shows that $\Psi$ is unital, hence contractive.

For a matrix unit $e_{k\ell}^{(n)}$ with $k\neq \ell$, by \Cref{lem:orthogonal-range-support-0exp} we have $\Phi(\varphi(e_{k\ell}^{(n)}))=0$. Using \Cref{lem:orthogonal-range-support-0exp} in the third step below, we get
\begin{align*}
\Psi(\varphi(e_{k\ell}^{(n)}))  &= \sum_{i,j=1}^{r_n}\Phi(\varphi(e_{k\ell}^{(n)})\cdot\varphi^\beta(e_{ji}^{(n)}))\cdot\varphi^\beta(e_{ij}^{(n)})\\
&= \sum_{i=1}^{r_n}\Phi(\varphi^{1+\beta}(e_{ki}^{(n)}))\cdot\varphi^\beta(e_{i\ell}^{(n)})\\
&=\varphi^{1+2\beta}(e_{k\ell}^{(n)}).
\end{align*}
Let $\pi\colon F\to \mathcal{M}(\mathrm{C}^\ast(\varphi(F)))$ denote the supporting $^*$-homomorphism of $\varphi$ \cite[Theorem~2.3]{WinZac09}, i.e.\ $\varphi(x)=\varphi(1_F)\pi(x)=\pi(x)\varphi(1_F)$ for all $x\in F$. For any $1\le k,\ell\le r_n$ and $1\le n\le N$ we have
\begin{align*}
\|\varphi(e_{k\ell}^{(n)})-\varphi^{1+2\beta}(e_{k\ell}^{(n)})\|&=\|\varphi(1_F)\pi(e_{k\ell}^{(n)})-\varphi(1_F)^{1+2\beta}\pi(e_{k\ell}^{(n)})\|\\
&\le \|\varphi(1_F)-\varphi(1_F)^{1+2\beta}\|\\
&<\varepsilon
\end{align*} 
by our choice of $\beta$. By the argument in the first paragraph of the proof, this establishes \ref{item:condexp3} and completes the proof for the unital case.

Assume now that $A$ is not unital, let $\eta>0$ and let $\mathcal{F}\Subset\mathrm{C}^\ast(D,\varphi(F))$. Then the pair of minimal unitizations $(D^\sim\subset A^\sim)$ is a unital $\mathrm{C}^\ast$-diagonal and $\varphi(\mc{N}_F(D_F))\subset\mc{N}_A(D)\subset\mc{N}_{A^\sim}(D^\sim)$, so by the unital case we obtain a faithful unital c.p.\ map $\bar\Psi\colon A^\sim\to\mathrm{C}^\ast(D^\sim,\varphi(F))$ satisfying $\bar\Psi(C)\subset C$ for all $D^\sim\subset C\subset A^\sim$, $\bar\Psi\vert_{D^\sim}=\mathrm{id}_{D^\sim}$ and $\|\bar\Psi(a)-a\|<\eta$ for all $a\in\mathcal{F}$. Let $(h_i)_i\subset D_+^1$ be an approximate unit for $A$. For $a\in A$, using that $D$ is a subset of the multiplicative domain of $\bar\Psi$, we have that
\[
\bar\Psi(a)=\lim_i\bar\Psi(h_ia)=\lim_ih_i\bar\Psi(a)\in A
\]
since $A$ is an ideal in $A^\sim$. It follows that the range of the c.p.c.\ map $\bar\Psi\vert_A$ is contained in $\mathrm{C}^\ast(D,\varphi(F))$, and if $D\subset B\subset A$ is an intermediate sub-$\mathrm{C}^\ast$-algebra, since $D^\sim\subset B^\sim\subset A^\sim$, we have that $\bar\Psi\vert_A(B) =\bar\Psi(B)\subset B^\sim \cap A=B$, as we wanted.
\end{proof}

\begin{example}
Allowing that $\Psi$ is only \emph{approximately} the identity on finite subsets of $\mathrm{C}^\ast(D,\varphi(F))$ is necessary. Consider for example $A=C(X)\otimes M_2$ for some infinite compact metric space $X$ and $D=C(X)\otimes D_2$. Let $x_0\in X$ be a point that is not isolated and take $h\in C(X)_+^1$ with $h(x_0)=0$ and $h(x)\ne0$ for $x\ne x_0$. The map $\varphi\colon M_2\to A$ given by $\varphi(z)=h\otimes z$ is c.p.c.\ and order zero, and $\varphi(\mc{N}_{M_2}(D_2))\subset\mc{N}_A(D)$. Moreover,
\[
\mathrm{C}^\ast(D,\varphi(M_2)) = \bigg\{\begin{bmatrix}f_{11}& f_{12}\\ f_{21} & f_{22}\end{bmatrix}: f_{ij}\in C(X) \text{ with }f_{12}(x_0)=f_{21}(x_0)=0.\bigg\}
\]
Now suppose there is a conditional expectation $\Psi\colon A\to \mathrm{C}^\ast(D,\varphi(M_2))$. For $f\in C(X)$, note that $\Psi(f\otimes e_{12})\cdot (1\otimes e_{11})= \Psi((f\otimes e_{12})\cdot(1\otimes e_{11}))=0$. For $x\in X\setminus\{x_0\}$ let $g\in C(X)$ be such that $g(x)=1$ and $g(x_0)=0$. Then $fg\otimes e_{12} = \Psi(fg\otimes e_{12}) = \Psi(f\otimes e_{12})\cdot (g\otimes e_{22})$. Evaluating at $x$ we obtain that $f(x)e_{12}=\Psi(f\otimes e_{12})(x)\cdot e_{22}$ and since this holds for all $x\ne x_0$ and $x_0$ is not isolated, we have that $f\otimes e_{12} = \Psi(f\otimes e_{12})\cdot (1\otimes e_{22})=\Psi(f\otimes e_{12})$. As $f$ was arbitrary this implies that $f(x_0)=0$ for all $f\in C(X)$, a contradiction.
\end{example}

\section{Intermediate subalgebras generated by a single normalizer-preserving order zero map and subhomogeneity}

\noindent In this section we establish \Cref{thmi:C}, which will be necessary for the main result of the paper. We begin with a lemma that will allow us to describe a large part of intermediate sub-$\mathrm{C}^\ast$-algebras by using ``coordinate systems", as discussed in the last part of the introduction.

\begin{lemma}\label{lem:ideal-matrix-subhom}
Let $(D\subset A)$ be a unital, abelian sub-$\mathrm{C}^\ast$-algebra, let $r\ge1$ and let $\varphi\colon M_r\to A$ be a linear c.p.c.\ order zero map with $\varphi(\mc{N}_{M_r}(D_r))\subset\mc{N}_A(D)$. If $J\coloneqq\langle\varphi(1_r)\rangle\trianglelefteq\mathrm{C}^\ast(D,\varphi(M_r))$, then there is an open set $U\subset\widehat{D}$ such that
\begin{equation}
J\cong C_0(U)\otimes M_r,
\end{equation}
with $D\cap J$ identified with $C_0(U)\otimes D_r$ under this isomorphism. 

Furthermore, if $D\subset B\subset \mathrm{C}^\ast(D,\varphi(M_r))$ is an intermediate sub-$\mathrm{C}^\ast$-algebra, there are (possibly empty) open subsets $V_{ij}\subset U$ satisfying
\begin{enumerate}[label=\normalfont{(\roman*)}]
\item $V_{ii}=U$ for all $1\le i\le r$,
\item $V_{ij}=V_{ji}$ for all $1\le i,j\le r$,
\item $V_{ij}\cap V_{jk}\subset V_{ik}$ for all $1\le i,j,k\le r$,
\end{enumerate}
and such that under the above isomorphism $B\cap J$ is identified with
\begin{equation}
\mathrm{span}\{f\otimes e_{ij}: f\in C_0(V_{ij}),\; 1\le i,j\le r\}.
\end{equation}
\end{lemma}

\begin{proof}
If $\varphi=0$ we can just take $U$ to be the empty set and there is nothing to show, so we assume that $\varphi$ is not zero. Set $C\coloneqq \mathrm{C}^\ast(D,\varphi(M_r))$ and note that $\varphi(1_r)$ is central in $C$, so 
\begin{equation}\label{eq:j-her}
J = \overline{\varphi(1_r)C\varphi(1_r)}
\end{equation}
and thus the sequence $\{\varphi(1_r)^\frac{1}{m}\}_{m=1}^\infty$ is an approximate unit for $J$. Let $\pi\colon M_r\to \mathcal{M}(\mathrm{C}^\ast(\varphi(M_r)))$ be the supporting $^*$-homomorphism of $\varphi$ and note that $\pi$ is unital (cf.\ the proof of \cite[Theorem~2.3]{WinZac09}). For $x\in M_r$ we have
\begin{align*}
\|\varphi(x)-\varphi(1_r)^\frac{1}{m}\varphi(x)\|&=\|(\varphi(1_r)-\varphi(1_r)^{1+\frac{1}{m}})\cdot\pi(x)\|\\
&\le \|\varphi(1_r)-\varphi(1_r)^{1+\frac{1}{m}}\| \cdot\|x\|\xrightarrow[m\to\infty]{}0
\end{align*}
which is to say that $\mathrm{C}^\ast(\varphi(M_r))\subset J$, and this inclusion is non-degenerate as we already noted that $\{\varphi(1_r)^\frac{1}{m}\}_{m=1}^\infty$ is an approximate unit for $J$. In particular, we have a unital embedding of multiplier algebras $\mc{M}(\mathrm{C}^\ast(\varphi(M_r)))\subset\mc{M}(J)$ (see \cite[Chapter 2]{Lance95}) so we can, and will, view $\pi$ as a unital embedding in $\mc{M}(J)$. Set $D_0\coloneqq \overline{\pi(e_{11})J\pi(e_{11})}\subset\mathcal{M}(J)$ and note that $D_0\subset J$ since $J\trianglelefteq\mc{M}(J)$, and moreover by \eqref{eq:j-her}
\[
D_0=\overline{\pi(e_{11})\varphi(1_r)C\varphi(1_r)\pi(e_{11})}=\overline{\varphi(e_{11})C\varphi(e_{11})}.
\]
By the $5^{\mathrm{th}}$ paragraph of the proof of \cite[Proposition~4.3]{KopWin24} we have that $D_0\subset D$, and thus $D_0=\overline{\varphi(e_{11})D\varphi(e_{11})}$. In particular, $D_0= C_0(U)$, where $U\subset\widehat{D}$ is the open support of $\varphi(e_{11})\in D$. Let $\iota\colon D_0\to \mathcal{M}(J)$ be the map
\begin{equation}\label{eq:subhomlemma-iota}
\iota(d) = \sum_{j=1}^r\pi(e_{j1})d\pi(e_{1j}),\quad d\in D_0.
\end{equation}
We have that $\iota$ is a $^\ast$-embedding, and arguing as above we note that $\iota(D_0)\subset D\cap J$ (in fact $\iota(D_0)=\overline{\varphi(1_r)D\varphi(1_r)}$). Moreover $[\iota(D_0),\pi(M_r)]=0$ as it can be directly verified that $\iota(D_0)$ commutes with the image of each matrix unit of $M_r$  under $\pi$. This induces an embedding $\bar\iota\colon D_0\otimes M_r \to\mc{M}(J)$ given by
\begin{equation}
\bar\iota(d\otimes x) = \iota(d)\pi(x),\quad d\in D_0,\; x\in M_r
\end{equation}
which in fact takes values in $J$. Also, since $C=\mathrm{C}^\ast(D,\varphi(M_r))$, it follows by \eqref{eq:j-her} that $\bar\iota(D_0\otimes M_r)=J$, and note also that $\bar\iota(D_0\otimes D_r) =\sum_{j=1}^r\overline{\varphi(e_{jj})D\varphi(e_{jj})} = D\cap J$. This establishes the first part of the lemma.


For the second part of the statement, it suffices to show that a $\mathrm{C}^\ast$-algebra $B$ with $C_0(U)\otimes D_r\subset B\subset C_0(U)\otimes M_r$ has the form
\[
B=\mathrm{span}\{f\otimes e_{ij}: f\in C_0(V_{ij}),\; 1\le i,j\le r\}
\]
for $V_{ij}\subset U$ open with $V_{ii}=U$, $V_{ij}=V_{ji}$ and $V_{ij}\cap V_{jk}\subset V_{ik}$ for all $i,j,k$. Note that if $b=\sum_{i,j=1}^rf_{ij}\otimes e_{ij}$ with $f_{ij}\in C_0(U)$, then $f_{ij}\otimes e_{ij}=\lim_k(u_k\otimes e_{ii})\cdot b\cdot(u_k\otimes e_{jj})\in B$, where $\{u_k\}_k\subset C_0(U)_+^1$ is an approximate unit for $C_0(U)$. Let $I_{ij}\coloneqq \{f\in C_0(U): f\otimes e_{ij}\in B\}$.  Clearly $I_{ij}$ is a closed linear subspace of $C_0(U)$ and $I_{ij}^*=I_{ji}$. Note that if $g\in C_0(U)$ we have that
\[
(fg)\otimes e_{ij}=(gf)\otimes e_{ij}=(g\otimes 1)\cdot(f\otimes e_{ij})\in B
\]
since $g\otimes 1\in C_0(U)\otimes D_r\subset B$; this shows that $I_{ij}$ is a closed two-sided ideal and thus it is also self-adjoint, whence $I_{ij}=I_{ji}$. Moreover, $I_{ii}=C_0(U)$, since $C_0(U)\otimes D_r\subset B$. We thus obtain open subsets $V_{ij}\subset U$ such that $I_{ij} = \{f\in C_0(U): f\vert_{U\setminus V_{ij}}=0\}=C_0(V_{ij})$. It follows from our earlier observations that $V_{ii}=U$ and that $V_{ij}=V_{ji}$ for all $i,j$.  Note also that $V_{ij}\cap V_{jk}\subset V_{ik}$: indeed, if this was not the case, then let $x\in V_{ij}\cap V_{jk}$ be a point that does not belong to $V_{ik}$ and find (by Urysohn's lemma) a function $f\in C_0(U)_+^1$ such that $f(x)=1$ and $\supp(f)\subset V_{ij}\cap V_{jk}$. Then $f\in I_{ij}\cap I_{jk}$. But note that $f\otimes e_{ik}=(f^{1/2}\otimes e_{ij})(f^{1/2}\otimes e_{jk})\in B$ so $f\in I_{ik}$, and thus $\supp(f)\subset V_{ik}$. This is a contradiction, since $f(x)=1$ and $x\not\in V_{ik}$. Finally, for $f\in I_{ij}$ we have that $f\otimes e_{ij}\in B$, so $B$ contains all elements of the form $f\otimes e_{ij}$ with $f\in C_0(V_{ij})$ and as mentioned earlier, any $b\in B$ is of the form $b=\sum_{i,j=1}^rf_{ij}\otimes e_{ij}$ with $f_{ij}\otimes e_{ij}\in B$, i.e.\ $f_{ij}\in I_{ij}=C_0(V_{ij})$ for all $i,j$.
\end{proof}

We note that an analogous statement of the preceding lemma holds also when the domain of the order zero map $\varphi$ is an arbitrary finite-dimensional $\mathrm{C}^\ast$-algebra $F$ (as opposed to a matrix algebra); indeed, write $F$ as a direct sum of matrix algebras $F\cong \bigoplus_{n=1}^NM_{r_n}$ and consider the ideals $J_n\coloneqq\langle\varphi(1_{M_{r_n}})\rangle \trianglelefteq\mathrm{C}^\ast(D,\varphi(F))$. Note that these are pairwise orthogonal (cf.\ the 3$^\mathrm{rd}$ paragraph in the proof of \cite[Proposition~4.3]{KopWin24}) and that $J_1+\dots +J_N=\langle\varphi(1_F)\rangle\trianglelefteq\mathrm{C}^\ast(D,\varphi(F))$. Applying \Cref{lem:ideal-matrix-subhom} to each $J_n$, we obtain the following corollary.

\begin{corollary}\label{cor:ideal-subhom}
Let $(D\subset A)$ be a unital, abelian sub-$\mathrm{C}^\ast$-algebra, let $F\coloneqq \bigoplus_{n=1}^NM_{r_n}$, let $D_F\coloneqq \bigoplus_{n=1}^ND_{r_n}$, and let $\varphi\colon F\to A$ be a linear c.p.c.\ order zero map such that $\varphi(\mc{N}_F(D_F))\subset\mc{N}_A(D)$. Set $J\coloneqq \langle \varphi(1_F)\rangle \trianglelefteq \mathrm{C}^\ast(D,\varphi(F))$. Then there exist open sets $U^{(1)},\dots, U^{(N)}\subset \widehat{D}$ such that
\begin{equation}\label{eq:identification-j}
(D\cap J \subset J)\cong \bigg(\bigoplus_{n=1}^N C_0(U^{(n)})\otimes D_{r_n}\subset \bigoplus_{n=1}^N C_0(U^{(n)})\otimes M_{r_n}\bigg).
\end{equation}
Each set $U^{(n)}$ is nonempty if and only if $\varphi$ is nonzero on the corresponding summand $M_{r_n}$ of $F$.

Furthermore, if $D\subset B\subset \mathrm{C}^\ast(D,\varphi(F))$ is an intermediate sub-$\mathrm{C}^\ast$-algebra, then there are open subsets $V_{ij}^{(n)}\subset U^{(n)}$ satisfying
\begin{enumerate}[label=\normalfont{(\roman*)}]
\item\label{item:setsVij-i} $V_{ii}^{(n)}=U^{(n)}$
\item\label{item:setsVij-ii} $V_{ij}^{(n)}=V_{ji}^{(n)}$
\item\label{item:setsVij-iii} $V_{ij}^{(n)}\cap V_{jk}^{(n)}\subset V_{ik}^{(n)}$
\end{enumerate}
for all $1\le i,j,k\le r_n$, $n=1,\dots,N$, and such that under the above isomorphism $B\cap J$ is identified with
\begin{equation}\label{eq:bcapj}
\bigoplus_{n=1}^N\mathrm{span}\{f\otimes e_{ij}^{(n)}: f\in C_0(V_{ij}^{(n)}),\; 1\le i,j\le r_n\}.
\end{equation}
\end{corollary}

\begin{remark}\label{rmk:subhomogeneity}
Denoting by $\pi\colon \mathrm{C}^\ast(D,\varphi(F))\to \mathrm{C}^\ast(D,\varphi(F))/J$ the quotient map, observe that $\pi(\mathrm{C}^\ast(D,\varphi(F)))=\pi(D)$, which is abelian. Since the quotient is abelian and by \Cref{cor:ideal-subhom} the ideal $J$ is subhomogeneous, $\mathrm{C}^\ast(D,\varphi(F))$ is also subhomogeneous, as an extension of subhomogeneous $\mathrm{C}^\ast$-algebras (and thus so is any sub-$\mathrm{C}^\ast$-algebra $B\subset\mathrm{C}^\ast(D,\varphi(F))$ \cite[Proposition~IV.1.4.3]{Bla06}). In the separable case, if we were only interested in obtaining finiteness of the nuclear dimension of this $\mathrm{C}^\ast$-algebra, since $(D\subset \mathrm{C}^\ast(D,\varphi(F)))$ is a diagonal by \cite{BrownExelFullerPittsReznikoff21, BrownExelFullerPittsReznikoff21corrigendum}, at this point we could use the main result of \cite{Kum86} to obtain a locally compact Hausdorff twist $\Sigma$ over a principal, locally compact Hausdorff {\'e}tale groupoid $\mc{G}$ such that $\mathrm{C}^\ast(D,\varphi(F))\cong \mathrm{C}^\ast_\mathrm{r}(\mc{G},\Sigma)$. Applying \cite[Theorem~B]{BonickeLi24} we have that $\mathrm{C}^\ast(D,\varphi(F))$ has finite nuclear dimension (see also \cite{Win04}). It is, however, necessary for us to keep track of the precise value of the more refined \emph{diagonal} dimension of $(D\subset\mathrm{C}^\ast(D,\varphi(F)))$, as well as its behavior on pairs with intermediate sub-$\mathrm{C}^\ast$-algebras, which we do in the rest of this section.
\end{remark}

The following proposition will allow us to obtain the special partition of unity discussed in the introduction that will be necessary for the proof of \Cref{thm:subhom}. For the convenience of the reader, let us briefly explain its statement. For $N=1$, it states that in a $d$-dimensional metric space that contains some disjoint homeomorphic copies of an open set $O$ and a compact $L\subset O$, one can find a $(d+1)$-colorable partition of unity $\mc{P}$ such that the functions of $\mc{P}$ supported away from the copies of $L$ have supports of small diameter, and the rest of $\mc{P}$ are functions that are \emph{in sync} along the copies of $O$: they arise from a partition of unity for $L$ in $C_c(O)_+^1$, replicated along the homeomorphic copies of $O$. While the support of a function of this latter type does not necessarily have small diameter, its intersection with each copy of $O$ does. For $N>1$ the statement is completely analogous, only now one has to take into account $N$ open sets (and their accompanying compact subsets) and their multiple homeomorphic copies, all being pairwise disjoint.

For the following proof, recall that for a surjection $q\colon X\to Y$ and $U\subset X$, we say that $U$ is \emph{saturated} for $q$ when $q^{-1}(q(U))=U$, i.e.\ $U$ contains every preimage of a point in $q(U)$. When $q$ is a quotient map between topological spaces and $U\subset X$ is open and saturated, $q(U)$ is open in $Y$ \cite{Munkres00}.

\begin{proposition}\label{prop:POU}
Let $X$ be a compact metric space with $\dim X\eqqcolon d\in\bb{N}$, let $N\ge1$, let $O^{(n)}$ be locally compact metric spaces and let $L^{(n)}\subset O^{(n)}$ be compact subsets for each $n=1,\dots,N$. Let $r_1,\dots,r_N\ge 1$ and let
\begin{equation}
\iota_r^{(n)}\colon O^{(n)}\to X,\quad 1\le r\le r_n,\; 1\le n\le N
\end{equation}
be continuous, open embeddings with pairwise disjoint images in $X$.

Then, for any $\delta>0$ there are finite families $\{\tilde{h}_s^{(n)}\}_{s\in S^{(n)}}\subset C_c(O^{(n)})_+^1$ and $\{g_t\}_{t\in T}\subset C(X)_+^1$ and a map $c\colon T\sqcup\bigsqcup_{n=1}^NS^{(n)}\to\{0,\dots,d\}$ such that
\begin{enumerate}[label=\normalfont{(\roman*)}]
\item\label{item:lemmaPOW-2} For $1\le n\le N$ each $\{\tilde{h}_s^{(n)}\}_{s\in S^{(n)}}$ is a partition of unity for $L^{(n)}$ with supports of diameter less than $\delta$ and $c\vert_{S^{(n)}}$ is a coloring map for it.

\item\label{item:lemmaPOW-1} The support of each $g_t$ has diameter less than $\delta$.

\item\label{item:lemmaPOW-3} With $h_s^{(n)}\in C_0\big(\bigsqcup_{r=1}^{r_n}\iota_r^{(n)}(O^{(n)})\big)_+^1$ defined as 
\begin{equation}\label{eq:def-POW-altogether}
h_s^{(n)}(\iota_r^{(n)}(z))=\tilde{h}_s^{(n)}(z),\quad z\in O^{(n)},\; 1\le r\le r_n,\;1\le n\le N
\end{equation}
and $h_s^{(n)}=0$ on $X\setminus \bigsqcup_{r=1}^{r_n}\iota_r^{(n)}(O^{(n)})$, the family
\[
\{g_t\}_{t\in T}\sqcup\{h_s^{(n)}:s\in S^{(n)}, 1\le n\le N\}
\]
is a partition of unity for $X$ and $c$ is a coloring map for it.
\end{enumerate}
\end{proposition}

\begin{proof}
Let $V_1^{(n)}, V_2^{(n)}\subset O^{(n)}$ be open sets with compact closures such that 
\[
L^{(n)}\subset V_1^{(n)}\subset \overline{V_1^{(n)}}\subset V_2^{(n)}\subset\overline{V_2^{(n)}}\subset O^{(n)}.
\]
Define a relation $\sim$ on $X$ by $x\sim x'$ if and only if $x=x'$ or there is $1\le n\le N$, $z\in \overline{V_2^{(n)}}$ and $1\le r, r'\le r_n$ such that $x=\iota_r^{(n)}(z)$ and $x'=\iota_{r'}^{(n)}(z)$. It is not hard to verify that $\sim$ is an equivalence relation on $X$. Note that each equivalence class of $\sim$ is a finite set (either a singleton, or a set with $r_1$ or $r_2$ or $\dots$ or $r_N$ elements) and thus a compact subset of $X$. Moreover, for $K\subset X$ closed we have that
\begin{equation}
\bigcup_{x\in K}[x]_\sim = K\cup \bigcup_{n=1}^N\bigcup_{r,s=1}^{r_n}\iota_r^{(n)}\bigg((\iota_s^{(n)})^{-1}\big(K\cap \iota_s^{(n)}\big(\overline{V_2^{(n)}}\big)\big)\bigg)
\end{equation}
which is closed, and so by \cite[Proposition~2.4.9]{Engelking89} and \cite[Theorem~4.2.13]{Engelking89} we have that the quotient space $Y\coloneqq X/\sim$ is a compact metrizable space. We claim that $\dim Y\le d$. To see this we will find a finite closed cover of $Y$ with sets of dimension at most $d$. For a tuple $\bar{r}\coloneqq \{\bar{r}_n\}_{n=1}^N\in\prod_{n=1}^N\{1,\dots, r_n\}$, put 
\begin{equation}
C_{\bar{r}}\coloneqq X\setminus\bigg(\bigcup_{n=1}^N\bigcup_{r\ne \bar{r}_n}\iota_r^{(n)}(O^{(n)})\bigg) \subset X
\end{equation}

and note that the sets $C_{\bar{r}}$ are closed (in particular, they all have dimension at most $d$ \cite[Theorem~3.1.4]{Engelking78}) and cover $X$. Moreover, $q$ is injective on each $C_{\bar{r}}$: indeed, let $x,x'\in C_{\bar{r}}$ be such that $q(x)=q(x')$. If $x\ne x'$, then there must exist some $1\le n\le N$ and some $1\le r,r'\le r_n$ such that for some $z\in \overline{V_2^{(n)}}$ we have that $x=\iota_r^{(n)}(z)$ and $x'=\iota_{r'}^{(n)}(z)$. Since $x,x'\in C_{\bar{r}}$, we must have that $r=r'=\bar{r}_n$ and thus $x=x'$, a contradiction. Since $q$ is injective on each $C_{\bar{r}}$ which is compact, the sets $q(C_{\bar{r}})$ are closed and homeomorphic to $C_{\bar{r}}$ and thus also have dimension at most $d$. Since these cover $Y$ by \cite[Theorem~3.1.8]{Engelking78} we have that $\dim Y$ is at most $d$.

We now set $F^{(n)}\coloneqq q(\iota_1^{(n)}(L^{(n)}))\subset Y$ for each $n=1,\dots, N$, which are compact subsets of $Y$ and note that $\{F^{(n)}\}_{n=1}^N$ are pairwise disjoint. 

Let $\delta>0$. For $z\in L^{(n)}$ choose an open neighborhood $D_z^{(n)}$ of $z$ with diameter less than $\delta$ and such that $\overline{D_z^{(n)}}\subset V_1^{(n)}$. Set
\begin{equation}
\hat{D}_z^{(n)}\coloneqq\bigsqcup_{r=1}^{r_n}\iota_r^{(n)}(D_z^{(n)}) \subset X,
\end{equation}
which is clearly (open and) saturated for $q$, so $q(\hat{D}_z^{(n)})$ is open in $Y$. 

For $y\in Y\setminus\bigsqcup_{n=1}^NF^{(n)}$, we claim that there is an open neighborhood $G_y\subset Y$ of $y$ such that $G_y$ is disjoint from $\bigsqcup_{n=1}^NF^{(n)}$ and $q^{-1}(G_y)$ is the disjoint union of finitely many open subsets of $X$ each of which has diameter less than $\delta$. Indeed, this is clearly the case if $y$ does not belong to $\bigsqcup_{n=1}^Nq\bigg(\iota_1^{(n)}(\overline{V_2^{(n)}})\bigg)$, since then $y$ has a unique preimage under $q$. If on the other hand $y=q(\iota_{1}^{(n)}(z))$ for some $z\in \overline{V_2^{(n)}}$ and $1\le n\le N$, then since $y\not\in F^{(n)}$ we have that $z\not\in L^{(n)}$. We can thus take an open neighborhood $\tilde{D}_z^{(n)}$ of $z$ with small enough diameter so that it does not intersect $L^{(n)}$ and each of the open sets $\iota_r^{(n)}(\tilde{D}_z^{(n)})$, $1\le r\le r_n$, has diameter less than $\delta$. Since $\bigsqcup_{r=1}^{r_n}\iota_r^{(n)}(\tilde{D}_z^{(n)})$ is (open and) saturated for $q$, setting $G_y\coloneqq q\big(\bigsqcup_{r=1}^{r_n}\iota_r^{(n)}(\tilde{D}_z^{(n)})\big)$ proves the claim.

Putting the two types of open neighborhoods together we obtain an open cover for $Y$, namely
\begin{equation}
\mathcal{U}\coloneqq \{q(\hat{D}_z^{(n)}):z\in L^{(n)}, 1\le n\le N\}\cup\bigg\{G_y: y\in Y\setminus\bigsqcup_{n=1}^NF^{(n)}\bigg\}.
\end{equation}
Since $Y$ is a compact metrizable space with $\dim Y\le d$, by \cite[Theorem~3.2.4]{Engelking78} there is a finite open cover $\mathcal{W}$ that refines $\mathcal{U}$ that is $(d+1)$-colorable, in the sense that there is a map $\hat{c}\colon \mathcal{W}\to \{0,\dots,d\}$ such that whenever $\hat{c}(W)=\hat{c}(W')$ we have that $W=W'$ or $W\cap W'= \emptyset$. For each $W\in\mathcal{W}$, pick $U_W\in\mathcal{U}$ such that $W\subset U_W$. Let $\{f_W\}_{W\in \mathcal{W}}\subset C(Y)_+^1$ be a partition of unity for $Y$ subordinate to $\mathcal{W}$. 

For $1\le n\le N$ set 
\[
S^{(n)}\coloneqq \big\{W\in\mathcal{W}: U_W=q(\hat{D}_z^{(n)})\text{ for some }z\in L^{(n)}\big\}
\]
and for $W\in S^{(n)}$ put $\tilde{h}_W^{(n)}\coloneqq f_W\circ q\circ \iota_1^{(n)}$. Note that since the closed support of $f_W$ is a subset of $U_W=q(\hat{D}_z^{(n)})$ (for some $z\in L^{(n)}$), the closed support of $\tilde{h}_W^{(n)}$ is a subset of $\overline{D_z^{(n)}}$ and in particular has diameter less than $\delta$ and is contained in a compact subset of $O^{(n)}$ (so $\tilde{h}_W^{(n)} \in C_c(O^{(n)})_+^1$). Note also that if $W\in\mc{W}$ is such that $W\not\in S^{(n)}$, then $U_W$ is disjoint from $F^{(n)}$, and so for $z\in L^{(n)}$ we have that $f_W(q(\iota_1^{(n)}(z)))=0$. Since $\{f_W\}_{W\in \mc{W}}$ is a partition of unity for $Y$, we obtain that $\sum_{W\in S^{(n)}}f_W(q(\iota_1^{(n)}(z)))=1$, and thus
\begin{equation}
\sum_{W\in S^{(n)}}\tilde{h}_W^{(n)}(z)=1.
\end{equation}
Modulo the statement about the coloring map (which we prove at the end), this establishes \ref{item:lemmaPOW-2}. 

For $1\le n\le N$ and $W\in S^{(n)}$, define $h_W^{(n)}\in C_0\big(\bigsqcup_{r=1}^{r_n}\iota_r^{(n)}(O^{(n)})\big)_+^1$ as in \eqref{eq:def-POW-altogether}, namely 
\begin{equation}
h_W^{(n)}(\iota_r^{(n)}(z))=\tilde{h}_W^{(n)}(z),\quad z\in O^{(n)},\; 1\le r\le r_n,\; 1\le n\le N
\end{equation}
and $h_W^{(n)}=0$ on $X\setminus\bigsqcup_{r=1}^{r_n}\iota_r^{(n)}(O^{(n)})$, and note that 
\begin{equation}\label{eq:notethistoo}
h_W^{(n)}=f_W\circ q.
\end{equation}
Indeed, both maps are $0$ on  $X\setminus \bigsqcup_{r=1}^{r_n}\iota_r^{(n)}(\overline{V_2^{(n)}})$, and if $x=\iota_r^{(n)}(z)$ for some $z\in \overline{V_2^{(n)}}$ and $1\le r\le r_n$, then $h_W^{(n)}(x)=f_W(q(\iota_1^{(n)}(z)))=f_W(q(\iota_r^{(n)}(z)))=f_W(q(x))$.

Now if $W\in\mc{W}\setminus\bigsqcup_{n=1}^NS^{(n)}$, then $U_W=G_y$ for some $y\in Y\setminus\bigsqcup_{n=1}^NF^{(n)}$, and we write $q^{-1}(U_W)=\bigsqcup_{j=1}^{m_W}D_j$ where $m_W\ge1$ and $D_j\subset X$ are open of diameter less than $\delta$.
Set
\begin{equation}
T\coloneqq\bigg\{(W,j): W\in\mc{W}\setminus\bigsqcup_{n=1}^NS^{(n)},\; 1\le j\le m_W\bigg\}
\end{equation}
and for $(W,j)\in T$ set
\begin{align*}
g_{(W,j)}(x)\coloneqq\begin{cases} f_W(q(x)),&\quad x\in D_j\\ 0,&\quad \text{else.}\end{cases}
\end{align*}
Note that $g_{(W,j)}$ is continuous: indeed, to prove this it suffices to see that $f_W(q(x))=0$ for $x\in\partial D_j$ and for that it is enough to see that $q(x)\not\in U_W$, i.e.\ that $x\not\in \bigsqcup_{i=1}^{m_W}D_i$. But indeed if $x\in D_i$ then $D_i\cap \partial D_j$ is nonempty and thus $D_i\cap D_j$ is nonempty, so $i=j$; but that contradicts the assumption that $x\in \partial D_j$. 
Also, the support of $g_{(W,j)}$ has diameter less than $\delta$, since $D_j$ does; this establishes \ref{item:lemmaPOW-1}. Note also that
\begin{equation}\label{eq:notethistooo}
\sum_{j=1}^{m_W}g_{(W,j)}=f_W\circ q.
\end{equation}
Putting together \eqref{eq:notethistoo} and \eqref{eq:notethistooo}, we see that $\{g_t\}_{t\in T}\sqcup\{h_W^{(n)}: W\in S^{(n)}, 1\le n\le N\}$ is a partition of unity for $X$. Finally, define the color map $c\colon T\sqcup\bigsqcup_{n=1}^NS^{(n)}\to\{0,\dots,d\}$ as $c(W,j)=\hat{c}(W)$ for $(W,j)\in T$ and $c(W)=\hat{c}(W)$ for $W\in\bigsqcup_{n=1}^NS^{(n)}$. Since for $j\ne j'$ we have that the supports of $g_{(W,j)}$ and $g_{(W,j')}$ are disjoint, $c$ is indeed a coloring map as we wanted.
\end{proof}

\begin{lemma}\label{lem:central-contraction}
Let A be a $\mathrm{C}^\ast$-algebra, and let $a\in A^1$, $b\in A$ and $z\in A_+^1$ be such that $a+b\in A^1$ and $z$ commutes with $a,b$. Then $a+zb\in A^1$.
\end{lemma}

\begin{proof}
Assume without loss of generality that $A$ is unital, and note that $0\le (1_A-z)a^*a+z(a+b)^*(a+b) - (a+zb)^*(a+zb)$, so
\begin{align*}
(a+zb)^*(a+zb)&\le (1_A-z)a^*a + z(a+b)^*(a+b)\\
&=(1_A-z)^\frac{1}{2}a^*a(1-z)^\frac{1}{2}+z^\frac{1}{2}(a+b)^*(a+b)z^\frac{1}{2}\\
&\le (1_A-z)+z=1_A
\end{align*}
as we wanted.
\end{proof}

\begin{theorem}\label{thm:subhom}
Let $(D\subset A)$ be a $\mathrm{C}^\ast$-diagonal with $D$ separable, let $(D_F\subset F)$ be a masa in a finite-dimensional $\mathrm{C}^\ast$-algebra, and let $\varphi\colon F\to A$ be a linear c.p.c.\ order zero map such that $\varphi(\mc{N}_F(D_F))\subset\mc{N}_A(D)$. Then, for any $\mathrm{C}^\ast$-algebra $D\subset B\subset \mathrm{C}^\ast(D,\varphi(F))$, we have that
\begin{equation}
\dim_{\mathrm{diag}}(D\subset B)=\dim\widehat{D}.
\end{equation}
\end{theorem}

\begin{proof}
If $\varphi=0$ then there is nothing to show, so we assume that $\varphi\ne0$. We first handle the unital case, in which $D\cong C(X)$ for some compact metric space $X$. Write $F= \bigoplus_{n=1}^NM_{r_n}$ and since all masas in finite-dimensional $\mathrm{C}^\ast$-algebras are conjugate, we can assume without loss of generality that $D_F = \bigoplus_{n=1}^ND_{r_n}$. Let $D\subset B\subset \mathrm{C}^\ast(D,\varphi(F))$ be an intermediate $\mathrm{C}^\ast$-algebra. Set $J\coloneqq\langle\varphi(1_F)\rangle\trianglelefteq\mathrm{C}^\ast(D,\varphi(F))$. Throughout the proof we will make repeated use of the identification of $J$ with $\bigoplus_{n=1}^NC_0(U^{(n)})\otimes M_{r_n}$ that was explicitly\footnote{To be more explicit, the proof of \Cref{lem:ideal-matrix-subhom} shows that $U^{(n)}\subset X$ is obtained as the open support of $\varphi(e_{11}^{(n)})$, and so it is identified with $U^{(n)}\times\{1\}$; in particular it is a locally compact metric space in its own right with the metric that $U^{(n)}\times\{1\}$ inherits from $X$ as its subspace.} described in \Cref{cor:ideal-subhom}, so we will treat this direct sum as a subalgebra of $\mathrm{C}^\ast(D,\varphi(F))$ (i.e.\ as $J$). We will also identify $B\cap J$ with the direct sum appearing in \eqref{eq:bcapj} with open sets $V_{ij}^{(n)}\subset U^{(n)}$ as in the second part of the statement of \Cref{cor:ideal-subhom}. Also, upon discarding matrix summands of $F$ on which $\varphi$ is zero (if any), we can assume without loss of generality that all the sets $U^{(n)}\subset X$ are nonempty. With this identification at hand, we have an open subset
\begin{equation}\label{eq:opensubset}
\bigsqcup_{n=1}^N\big(U^{(n)}\times\{1,\dots,r_n\}\big)\subset X
\end{equation}
which can be viewed as the spectrum of $D\cap J$. Consider also the open embeddings $\bar\iota_r^{(n)}\colon U^{(n)}\to X$ defined as 
\begin{equation}\label{eq:embeddings}
\bar\iota_r^{(n)}(x)=(x,r),\quad x\in U^{(n)},\; 1\le r\le r_n,\; 1\le n\le N.
\end{equation}
We remark that an element $d\in D$ belongs to $C_0(U^{(n)})\otimes e_{rr}^{(n)}$ precisely when its open support is contained in $\bar\iota_r^{(n)}(U^{(n)})=U^{(n)}\times\{r\}$. In this case, $d$ is identified with $(d\circ\bar\iota_r^{(n)})\otimes e_{rr}^{(n)}$. On the other hand, given $f\in C_0(U^{(n)})$, the elementary tensor $f\otimes e_{rr}^{(n)}$ is identified with the function in $C(X)$ defined to be zero everywhere outside $\bar\iota_r^{(n)}(U^{(n)})$ and defined to be $f(y)$ on the point $\bar\iota_r^{(n)}(y)\in X$.

Since $d\coloneqq \dim X=\dim_\mathrm{nuc}(D)\le \dim_\mathrm{diag}(D\subset B)$ by \cite[Proposition~2.4]{WinZac10} and \cite[Remarks~2.2]{LiLiaWin23}, it suffices to show the reverse inequality, which we shall establish by verifying the criterion in \cite[Proposition~2.3]{LiLiaWin23}. To that end let $\varepsilon>0$ and let $\mc{F}\Subset B^1$ (since we are in the unital case, we take $1_A$ as $h$ in \cite[Proposition~2.3]{LiLiaWin23}). If $d=\infty$ there is nothing to show so assume that $d\in\mathbb{N}$. Let $0<\eta<(d+1)^{-1}\cdot (\max_{1\le n\le N}r_n^2 + 5)^{-1}\cdot \varepsilon$ and apply \Cref{prop:cond-exps} with this $\eta>0$ and this $\mathcal{F}$.  By \eqref{eq:cond-exp-def}, each $\Psi(b)$ is of the form $\Psi(b) =  d_b + e_b$ where $d_b\in D^1$ and $e_b\in (B\cap J)^1$. Let $z\in (\bigoplus_{n=1}^N C_c(U^{(n)})\otimes 1_{r_n})_+^1\subset D\cap J$ be such that $\|ze_b-e_b\|<\eta$ for all $b\in\mathcal{F}$. Note also that $z$ is central in $\mathrm{C}^\ast(D,\varphi(F))$, and so by \Cref{lem:central-contraction} $d_b+ze_b$ are contractions with 
\begin{equation}
\|d_b+ze_b - b\|<2\eta,\quad b\in\mathcal{F}.
\end{equation}
Moreover, $ze_b$ is a contraction in $ (B\cap J)\cap \bigoplus_{n=1}^NC_c(U^{(n)})\otimes M_{r_n}$, and so $ze_b$ is a contraction in
\begin{equation*}
\bigoplus_{n=1}^N\mathrm{span}\{f\otimes e_{ij}^{(n)}: f\in C_c(V_{ij}^{(n)}),\; 1\le i,j\le r_n\}.
\end{equation*}
For each $b\in\mc{F}$ write $ze_b = \bigoplus_{n=1}^N[b_{ij}^{(n)}]_{i,j=1}^{r_n}$ with $b_{ij}^{(n)}\in C_c(V_{ij}^{(n)})^1$ and take compact sets $K_{ij}^{(n)}\subset V_{ij}^{(n)}$ such that $\bigcup_{b\in\mathcal{F}}\mathrm{supp}(b_{ij}^{(n)})\subset K_{ij}^{(n)}$. 

Let $K^{(n)}\coloneqq\bigcup_{i,j=1}^{r_n}K_{ij}^{(n)} \subset U^{(n)}$; while some of the sets $V_{ij}^{(n)}$ (and thus also some of $K_{ij}^{(n)}$) may be empty, we can always find nonempty open sets $W^{(n)}$ and $O^{(n)}$ with compact closures such that 
\begin{equation}\label{eq:subsets-chain}
K^{(n)}\subset W^{(n)} \subset\overline{W^{(n)}}\subset O^{(n)}\subset \overline{O^{(n)}}\subset U^{(n)}.
\end{equation}
For $1\le n\le N$, $b\in\mc{F}$ and $y\in U^{(n)}$, set
\begin{equation}\label{eq:mb}
m_b^{(n)}(y)\coloneqq\mathrm{diag}(d_b(y,1),\dots,d_b(y,r_n))+ [b_{ij}^{(n)}(y)]_{i,j=1}^{r_n}\in M_{r_n}.
\end{equation}
Note that each $m_b^{(n)}(y)$ is a contraction in $M_{r_n}$. Indeed, for $y\in U^{(n)}$ let $a\in C_c(U^{(n)})_+^1$ be such that $a(y)=1$ and consider $(a\otimes 1_{r_n})\cdot (d_b+ze_b)\cdot (a\otimes 1_{r_n})$ which is a contraction in $C_0(U^{(n)})\otimes M_{r_n}$; viewing this as an $M_{r_n}$-valued function on $U^{(n)}$ and evaluating at $y$ yields a contraction in $M_{r_n}$ which can be directly seen to be equal to $m_b^{(n)}(y)$, using our remark after \eqref{eq:embeddings}.

Choose $\delta>0$ small enough so that all of the following hold:
\begin{enumerate}[label=(\alph*)]
\item\label{item:delta-a}  For each $1\le n\le N$, $1\le i,j\le r_n$, whenever $K_{ij}^{(n)}$ is nonempty, the $\delta$-neighborhood of $K_{ij}^{(n)}$ (that is, the set of points at distance less than $\delta$ from $K_{ij}^{(n)}$) is contained in $V_{ij}^{(n)}$.
\item\label{item:delta-b} If $y_1,y_2\in X$ are at distance at most $\delta$, then $|d_b(y_1)-d_b(y_2)|<\eta$ for all $b\in\mc{F}$.
\item\label{item:delta-c} For each $1\le n\le N$, if $y_1,y_2\in \overline{O^{(n)}}$ are at distance at most $\delta$, then $|d_b(y_1,r)-d_b(y_2,r)|<\eta$ for all $1\le r\le r_n$ and $b\in\mc{F}$.
\item\label{item:delta-d} For each $1\le n\le N$, if $y_1,y_2\in \overline{O^{(n)}}$ are at distance at most $\delta$, then  $|b_{ij}^{(n)}(y_1)-b_{ij}^{(n)}(y_2)|<\eta$ for all $b\in\mc{F}$, $i,j=1,\dots,r_n$.
\end{enumerate}
Apply \Cref{prop:POU} with the open embeddings $\iota_r^{(n)}\coloneqq \bar\iota_r^{(n)}\vert_{O^{(n)}}$ and $\overline{W^{(n)}}$ as the sets $L^{(n)}$ in that statement. For this $\delta>0$ we obtain finite families $\{\tilde{h}_s^{(n)}\}_{s\in S^{(n)}}\subset C_c(O^{(n)})_+^1$, $n=1,\dots N$, and $\{g_t\}_{t\in T}\subset C(X)_+^1$ together with a map $c\colon T\sqcup\bigsqcup_{n=1}^NS^{(n)}\to\{0,\dots,d\}$ satisfying \ref{item:lemmaPOW-2}--\ref{item:lemmaPOW-3} as in \Cref{prop:POU}. Moreover, by discarding those functions $g_t$ or $\tilde{h}_s^{(n)}$ that are identically zero, we can assume without loss of generality that the support of each $g_t$ and each $\tilde{h}_s^{(n)}$ (and thus of each $h_s^{(n)}$ as well) is nonempty.

For $s\in S^{(n)}$ let $\sim_s^{(n)}$ be the relation on $\{1,\dots,r_n\}$ defined by
\begin{equation}\label{eq:equivalence}
i\sim_s^{(n)}j\iff \supp(\tilde{h}_s^{(n)})\subset V_{ij}^{(n)},\quad 1\le i,j\le r_n.
\end{equation}
Since $V_{ii}^{(n)}=U^{(n)}$ and $V_{ij}^{(n)}=V_{ji}^{(n)}$ by \ref{item:setsVij-i} and \ref{item:setsVij-ii} of \Cref{cor:ideal-subhom} respectively, $\sim_s^{(n)}$ is symmetric and reflexive. Now if $i\sim_s^{(n)}j\sim_s^{(n)}k$, then using \ref{item:setsVij-iii} of \Cref{cor:ideal-subhom} we have that
\begin{equation}
\supp(\tilde{h}_s^{(n)})\subset V_{ij}^{(n)}\cap V_{jk}^{(n)}\subset V_{ik}^{(n)}
\end{equation}
so $i\sim_s^{(n)}k$ and thus $\sim_s^{(n)}$ is also transitive and hence an equivalence relation.

Let $Q_s^{(n)}$ denote the finite set of equivalence classes of $\sim_s^{(n)}$, and for $Q\in Q_s^{(n)}$ let $p_Q\coloneqq\sum_{i\in Q}e_{ii}^{(n)}\in M_{r_n}$ and set
\begin{equation}\label{eq:M_Q}
M_Q\coloneqq p_QM_{r_n}p_Q\subset M_{r_n}, \text{ and } D_Q\coloneqq p_QD_{r_n}p_Q\subset D_{r_n}.
\end{equation}
For $t\in T$ we define c.p.\ linear maps $\vartheta_t\colon\bb{C}\to D$ given by
\begin{equation}\label{eq:theta-t}
\vartheta_t(\lambda)\coloneqq \lambda g_t,\quad \lambda\in\bb{C}.
\end{equation}
For $1\le n\le N$, $s\in S^{(n)}$ and $Q\in Q_s^{(n)}$ we define c.p.\ linear maps $ \vartheta_{s,Q}^{(n)}\colon M_Q\to J$ given by
\begin{equation}\label{eq:theta-s-q-n}
\vartheta_{s,Q}^{(n)}(x)\coloneqq\tilde{h}_s^{(n)}\otimes x,\quad x\in M_Q.
\end{equation}
Note that each $\vartheta_{s,Q}^{(n)}$ in fact takes values in $B\cap J$. Indeed, if $a=[\lambda_{ij}]_{i,j=1}^{r_n}\in M_Q$ and $i,j$ are such that $\lambda_{ij}\ne0$, then $i,j\in Q$ so $i\sim_s^{(n)}j$, and thus $\lambda_{ij}\tilde{h}_s^{(n)}\in C_0(V_{ij}^{(n)})$ by \eqref{eq:equivalence}, which shows that $\vartheta_{s,Q}^{(n)}(a)\in B\cap J$ by \eqref{eq:bcapj}.

For $0\le c\le d$ let $S_c^{(n)}=\{s\in S^{(n)}: c(s)=c\}$ and let $T_c=\{t\in T: c(t)=c\}$. Consider the finite-dimensional $\mathrm{C}^\ast$-algebra
\begin{equation}\label{eq:Ec}
E^{(c)}\coloneqq \bigoplus_{t\in T_c}\bb{C} \oplus \bigoplus_{n=1}^N\bigoplus_{s\in S^{(n)}_c}\bigoplus_{Q\in Q_s^{(n)}}M_Q,
\end{equation}
and the masa $(D_{E^{(c)}}\subset E^{(c)})$ defined as
\begin{equation}\label{eq:Dc}
D_{E^{(c)}}\coloneqq \bigoplus_{t\in T_c}\bb{C}\oplus\bigoplus_{n=1}^N\bigoplus_{s\in S^{(n)}_c}\bigoplus_{Q\in Q_s^{(n)}}D_Q.
\end{equation}
Let $\vartheta^{(c)}\colon E^{(c)}\to B$ be the c.p.\ linear map defined as
\begin{equation}\label{eq:theta-c}
\vartheta^{(c)}=\sum_{t\in T_c}\vartheta_t+\sum_{n=1}^N\sum_{s\in S_c^{(n)}}\sum_{Q\in Q_s^{(n)}}\vartheta_{s,Q}^{(n)}.
\end{equation}
We first claim that $\vartheta^{(c)}$ is contractive and order zero. Indeed, it is clear by \eqref{eq:theta-t} and \eqref{eq:theta-s-q-n} that each of the maps $\vartheta_t$ and $\vartheta_{s,Q}^{(n)}$ are order zero. Observe also that the maps $\{\vartheta_t\}_{t\in T_c}\sqcup\{\vartheta_{s,Q}^{(n)}: Q\in Q_s^{(n)},\; s\in S_c^{(n)},\; 1\le n\le N\}$  have pairwise orthogonal ranges, and so $\vartheta^{(c)}$ is order zero. Moreover, we have that
\begin{align}\label{eq:ev-at-1}
\vartheta^{(c)}(1_{E^{(c)}}) &=\sum_{t\in T_c}g_t + \sum_{n=1}^N\sum_{s\in S_c^{(n)}}\sum_{Q\in Q_s^{(n)}}\tilde{h}_s^{(n)}\otimes p_Q\\
&=\sum_{t\in T_c}g_t + \sum_{n=1}^N\sum_{s\in S_c^{(n)}}\tilde{h}_s^{(n)}\otimes 1_{M_{r_n}}\nonumber\\
&=\sum_{t\in T_c}g_t+\sum_{n=1}^N\sum_{s\in S_c^{(n)}}h_s^{(n)}\nonumber
\end{align}
which is a contraction as a sum of pairwise orthogonal contractions. This shows that $\vartheta^{(c)}$ is contractive. 

We claim that $\vartheta^{(c)}(\mathcal{N}_{E^{(c)}}(D_{E^{(c)}}))\subset\mathcal{N}_A(D)$. Indeed, by \cite[Lemma~1.4]{LiaTik22}
it suffices to see that, individually, each of the maps $\vartheta_t$ and $\vartheta_{s,Q}^{(n)}$ preserve normalizers, since all such maps appearing in the right side of \eqref{eq:theta-c} have pairwise orthogonal ranges. For $\vartheta_t$ this is immediate since its range lies in $D$. For $\vartheta_{s,Q}^{(n)}$, it follows by \eqref{eq:theta-s-q-n} and the identification \eqref{eq:identification-j} that $\vartheta_{s,Q}^{(n)}(\mc{N}_{M_Q}(D_Q))\subset\mc{N}_J(D\cap J)$, and since $D\cap J$ contains an approximate unit for $J$, we have $\mc{N}_J(D\cap J)\subset\mc{N}_A(D)$.

We now turn to the necessary approximations. Set $x_{1_A}^{(c)}\coloneqq 1_{E^{(c)}}$ and $x_{1_A}\coloneqq\bigoplus_{c=0}^dx_{1_A}^{(c)}$. We will show that for each $b\in\mc{F}$ there exist contractions $x_b^{(c)}\in (E^{(c)})^1$ such that, with
\begin{equation}\label{eq:xb}
x_b\coloneqq\bigoplus_{c=0}^dx_b^{(c)}\in \bigg(\bigoplus_{c=0}^dE^{(c)}\bigg)^1,\quad b\in\mc{F},
\end{equation}
we have that
\begin{equation}\label{eq:wanted1-sum-approx}
\bigg\|\sum_{c=0}^d\vartheta^{(c)}(x_b^{(c)})-b\bigg\|<\varepsilon,\quad b\in\mc{F}\cup\{1_A\}
\end{equation}
and
\begin{equation}\label{eq:wanted2-summand-approx}
\|\vartheta^{(c)}(x_{1_A}^{(c)})\cdot b-\vartheta^{(c)}(x_b^{(c)})\|<\varepsilon/(d+1),\quad 0\le c\le d,\; b\in\mc{F}.
\end{equation}
In fact, by \eqref{eq:ev-at-1} and the fact that $\{g_t\}_{t\in T}\sqcup\{h_s^{(n)}:s\in S^{(n)},\; 1\le n\le N\}$ is a partition of unity for $X$, we see that $\sum_{c=0}^d\vartheta^{(c)}(x_{1_A}^{(c)})=1_A$, and so it suffices to only verify \eqref{eq:wanted2-summand-approx}, since \eqref{eq:wanted1-sum-approx} follows by \eqref{eq:wanted2-summand-approx} and a triangle inequality.

To that end, for $t\in T$ take $y_t\in \supp(g_t)$ and for $1\le n\le N$ and $s\in S^{(n)}$ take $y_s^{(n)}\in \supp(\tilde{h}_s^{(n)})$. For $b\in\mc{F}$ we define
\begin{equation}\label{eq:xbc}
x_b^{(c)}\coloneqq \bigoplus_{t\in T_c}d_b(y_t)\oplus\bigoplus_{n=1}^N\bigoplus_{s\in S_c^{(n)}}\bigoplus_{Q\in Q_s^{(n)}}p_Q\cdot m_b^{(n)}(y_s^{(n)})\cdot p_Q\in E^{(c)}
\end{equation}
which are contractions, since $d_b\in D^1$ and $m_b^{(n)}(y)$ is a contraction for all $y\in U^{(n)}$ and all $1\le n\le N$. In order to verify \eqref{eq:wanted2-summand-approx}, note first that
\begin{equation}\label{eq:gtzeb=0}
g_t\perp ze_b,\quad t\in T,\; b\in\mc{F}.
\end{equation}
Indeed, we have that $ze_b=\bigoplus_{n=1}^N[b_{ij}^{(n)}]_{i,j=1}^{r_n}$ and $\supp(b_{ij}^{(n)})\subset K^{(n)}$. Now 
\begin{align*}
\sum_{n=1}^N\sum_{s\in S^{(n)}}(\tilde{h}_s^{(n)}\otimes 1_{r_n})\cdot ze_b &= \bigoplus_{n=1}^N\bigg[\sum_{s\in S^{(n)}}\tilde{h}_s^{(n)}b_{ij}^{(n)}\bigg]_{i,j=1}^{r_n}\\
&=\bigoplus_{n=1}^N[b_{ij}^{(n)}]_{i,j=1}^{r_n}\\
&=ze_b
\end{align*}
where we used the fact that $\{\tilde{h}_s^{(n)}\}_{s\in S^{(n)}}$ is a partition of unity for $\overline{W^{(n)}}\supset K^{(n)}$. Likewise we have that $ze_b\cdot \sum_{n=1}^N\sum_{s\in S^{(n)}}(\tilde{h}_s^{(n)}\otimes 1_{r_n})=ze_b$, and thus $\sum_{t\in T}g_tze_b=ze_b\sum_{t\in T}g_t=0$. It follows that $g_tze_b=ze_bg_t=0$ for all $t\in T$. Now using \eqref{eq:ev-at-1} and \eqref{eq:gtzeb=0} we have that
\begin{align}\label{eq:step1}
&\|\vartheta^{(c)}(1_{E^{(c)}})\cdot (d_b+ze_b)-\vartheta^{(c)}(x_b^{(c)})\|= \bigg\|\sum_{t\in T_c}g_t\cdot (d_b-d_b(y_t)1_A)+\\
&\sum_{n=1}^N\sum_{s\in S_c^{(n)}}\bigg((\tilde{h}_s^{(n)}\otimes 1_{r_n})\cdot (d_b+ze_b) - \sum_{Q\in Q_s^{(n)}}\tilde{h}_s^{(n)}\otimes p_Qm_b^{(n)}(y_s^{(n)})p_Q\bigg)\bigg\|.\nonumber
\end{align}
Using \ref{item:delta-b} together with \ref{item:lemmaPOW-1} of \Cref{prop:POU}, we note that
\begin{align}\label{eq:step2}
\bigg\|\sum_{t\in T_c}(g_t\cdot d_b-d_b(y_t)g_t)\bigg\|&=\max_{t\in T_c}\sup_{y\in \supp(g_t)}|d_b(y)g_t(y)-d_b(y_t)g_t(y)|\\
&\le\max_{t\in T_c}\sup_{y\in \supp(g_t)}|d_b(y)-d_b(y_t)|\nonumber\\
&<\eta.\nonumber
\end{align}
Next, by orthogonality, we have that
\begin{align}\label{eq:step3}
&\bigg\|\sum_{n=1}^N\sum_{s\in S_c^{(n)}}\bigg((\tilde{h}_s^{(n)}\otimes 1_{r_n})\cdot (d_b+ze_b) - \sum_{Q\in Q_s^{(n)}}\tilde{h}_s^{(n)}\otimes p_Qm_b^{(n)}(y_s^{(n)})p_Q\bigg)\bigg\|=\\
&\max_{1\le n\le N}\max_{s\in S_c^{(n)}}\bigg\|(\tilde{h}_s^{(n)}\otimes 1_{r_n})\cdot (d_b+ze_b) - \sum_{Q\in Q_s^{(n)}}\tilde{h}_s^{(n)}\otimes p_Qm_b^{(n)}(y_s^{(n)})p_Q\bigg\|.\nonumber
\end{align}
Fix $1\le n\le N$ and $s\in S_c^{(n)}$. Inspecting \eqref{eq:mb}, we observe that the sum $\sum_{Q\in Q_s^{(n)}}\tilde{h}_s^{(n)}\otimes p_Qm_b^{(n)}(y_s^{(n)})p_Q$ can be written as
\begin{align*}
\tilde{h}_s^{(n)}\otimes\mathrm{diag}(d_b(y_s^{(n)},1),\dots,d_b(y_s^{(n)},r_n)) + \tilde{h}_s^{(n)}\otimes \sum_{Q\in Q_s^{(n)}} p_Q[b_{ij}^{(n)}(y_s^{(n)})]_{i,j=1}^{r_n}p_Q.
\end{align*}
Now by our remark after \eqref{eq:embeddings}, $(\tilde{h}_s^{(n)}\otimes 1_{r_n})\cdot d_b$ is identified with the diagonal matrix-valued function $\mathrm{diag}(\tilde{h}_s^{(n)}\cdot (d_b\circ\iota_1^{(n)}),\dots,\tilde{h}_s^{(n)}\cdot (d_b\circ\iota_{r_n}^{(n)}))$, and so
\begin{align}\label{eq:step4}
&\|(\tilde{h}_s^{(n)}\otimes 1_{r_n})\cdot d_b - \tilde{h}_s^{(n)}\otimes \mathrm{diag}(d_b(y_s^{(n)},1),\dots,d_b(y_s^{(n)},r_n))\|=\\
&\max_{1\le r\le r_n}\sup_{y\in \supp(\tilde{h}_s^{(n)})}|(d_b(y,r)-d_b(y_s^{(n)},r))\cdot \tilde{h}_s^{(n)}(y)|<\eta\nonumber
\end{align}
by \ref{item:lemmaPOW-2} of \Cref{prop:POU} together with \ref{item:delta-c}. Next, note that
\begin{align}\label{eq:step5}
&\bigg\|(\tilde{h}_s^{(n)}\otimes 1_{r_n})\cdot ze_b - \tilde{h}_s^{(n)}\otimes\sum_{Q\in Q_s^{(n)}} p_Q[b_{ij}^{(n)}(y_s^{(n)})]_{i,j=1}^{r_n}p_Q\bigg\| = \\
\sup_{y\in O^{(n)}}&\bigg\|[\tilde{h}_s^{(n)}(y)b_{ij}^{(n)}(y)]_{i,j=1}^{r_n}-\sum_{Q\in Q_s^{(n)}}p_Q[\tilde{h}_s^{(n)}(y)b_{ij}^{(n)}(y_s^{(n)})]_{i,j=1}^{r_n}p_Q\bigg\|_{M_{r_n}}.\nonumber
\end{align}
We claim that $[\tilde{h}_s^{(n)}(y)b_{ij}^{(n)}(y)]_{i,j=1}^{r_n}=\sum_{Q\in Q_s^{(n)}}p_Q[\tilde{h}_s^{(n)}(y)b_{ij}^{(n)}(y)]_{i,j=1}^{r_n}p_Q$ for all $y\in O^{(n)}$. To see this it suffices to observe that whenever the $(i,j)$-entry of the matrix on the left side of the equation is non-zero, then $i\sim_s^{(n)}j$. This is indeed the case, since if $\tilde{h}_s^{(n)}(y)b_{ij}^{(n)}(y)\ne0$, then $y\in\supp(\tilde{h}_s^{(n)})\cap K_{ij}^{(n)}$, and as $\supp(\tilde{h}_s^{(n)})$ is a subset of $O^{(n)}$ that has diameter less than $\delta$, by \ref{item:delta-a} this implies that $\supp(\tilde{h}_s^{(n)})\subset V_{ij}^{(n)}$ and thus $i\sim_s^{(n)}j$. Together with \eqref{eq:step5}, we obtain
\begin{align}\label{eq:step6}
&\bigg\|(\tilde{h}_s^{(n)}\otimes 1_{r_n})\cdot ze_b - \tilde{h}_s^{(n)}\otimes\sum_{Q\in Q_s^{(n)}} p_Q[b_{ij}^{(n)}(y_s^{(n)})]_{i,j=1}^{r_n}p_Q\bigg\|=\\
&\sup_{y\in O^{(n)}}\max_{Q\in Q_s^{(n)}}\big\|p_Q\big[\tilde{h}_s^{(n)}(y)\cdot \big(b_{ij}^{(n)}(y)-b_{ij}^{(n)}(y_s^{(n)})\big)\big]_{i,j=1}^{r_n}p_Q\big\|\le \nonumber\\
&r_n^2\cdot \sup_{y\in \supp(\tilde{h}_s^{(n)})}\max_{Q\in Q_s^{(n)}}\max_{1\le i,j\le r_n}|b_{ij}^{(n)}(y)-b_{ij}^{(n)}(y_s^{(n)})|<r_n^2\cdot \eta,\nonumber
\end{align}
where we used \ref{item:delta-d} and that $\supp(\tilde{h}_s^{(n)})$ is a subset of $O^{(n)}$ with diameter less than $\delta$. Putting everything together and applying the triangle inequality, we obtain
\begin{equation}\label{eq:final}
\|\vartheta^{(c)}(x_{1_A}^{(c)})\cdot b-\vartheta^{(c)}(x_b^{(c)})\|<(\max_{1\le n\le N}r_n^2+5)\cdot \eta<\varepsilon/(d+1)
\end{equation}
for all $b\in\mc{F}$ and $0\le c\le d$, as we wanted. This completes the proof for the unital case.

Assume now that $A$ is non-unital, and let $D\subset B\subset \mathrm{C}^\ast(D,\varphi(F))$ be an intermediate sub-$\mathrm{C}^\ast$-algebra. By the unital case (cf.\ the first lines of the last paragraph in the proof of \Cref{prop:cond-exps}), we get that $\dim_\mathrm{diag}(D^\sim\subset B^\sim)=\dim\widehat{D^\sim}$. By \cite[Theorem~3.1~(iv)]{LiLiaWin23} and \cite[Remark~2.11]{WinZac10} it follows that $\dim_\mathrm{diag}(D\subset B)=\dim\widehat{D}$ as we wanted.
\end{proof}

Combining \Cref{thm:subhom} together with \Cref{rmk:subhomogeneity} yields \Cref{thmi:C}.
\section{Finite diagonal dimension passes to intermediate subalgebras}

\noindent We have now gathered all the necessary ingredients to prove \Cref{thmi:A}.

\begin{theorem}\label{thm:main}
Let $(D\subset A)$ be a $\mathrm{C}^\ast$-diagonal with $D$ separable. If $D\subset B\subset A$ is an intermediate sub-$\mathrm{C}^\ast$-algebra, then
\begin{equation}\label{eq:main}
\dim_{\mathrm{diag}}^{+1}(D\subset B)\le \dim^{+1}\widehat{D}\cdot\dim^{+1}_\mathrm{diag}(D\subset A).
\end{equation}
\end{theorem}

\begin{proof}
If the pair $(D\subset A)$ has infinite diagonal dimension there is nothing to show, so let us assume that $\dim_\mathrm{diag}(D\subset A)\eqqcolon d\in\mathbb{N}$, and thus also $\widehat{D}$ has finite covering dimension. Let $\varepsilon>0$, let $\mc{F}\Subset B^1$, and let $h\in D_+^1$ be such that $hb=bh=b$ for all $b\in\mc{F}$.  Let $0<\eta< (5d+6)^{-1}\varepsilon$. Let $\hat{\phi},\bar{\phi}\in C_0((0,1])_+^1$ be such that $\|\hat{\phi}-\mathrm{id}_{(0,1]}\|<\eta$, $\hat{\phi}(1)=\bar{\phi}(1)=1$ and $\bar{\phi}\hat{\phi}=\hat{\phi}$. Set $\hat{h}\coloneqq \hat{\phi}(h)$, $\bar{h}\coloneqq\bar{\phi}(h)\in \mathrm{C}^\ast(h)_+^1$ and note that $\|\hat{h}-h\|<\eta$, $\bar{h}\hat{h} = \hat{h}$, and $\bar{h}b=b\bar{h}=b$ and $\hat{h}b=b\hat{h}=b$ for all $b\in\mc{F}$.

By \cite[Proposition~2.3]{LiLiaWin23} we can find c.p.c.\ order zero maps $\varphi^{(i)}\colon F^{(i)}\to A$, $i=0,\dots,d$, with each $F^{(i)}$ finite-dimensional, together with masas $D_{F^{(i)}}\subset F^{(i)}$ such that $\varphi^{(i)}(\mc{N}_{F^{(i)}}(D_{F^{(i)}}))\subset\mc{N}_A(D)$ for all $i=0,\dots,d$, and such that there exist elements $x_b=x_b^{(0)}\oplus\dots\oplus x_b^{(d)}\in (\bigoplus_{i=0}^dF^{(i)})^1$ for $b\in\mc{F}\cup\{\hat{h}\}$ and $x_{\bar{h}}=x_{\bar{h}}^{(0)}\oplus\dots\oplus x_{\bar{h}}^{(d)} \in (\bigoplus_{i=0}^dD_{F^{(i)}})^1$ such that
\begin{equation}\label{eq:mainproof1}
\bigg\|b-\sum_{i=0}^d\varphi^{(i)}(x_b^{(i)})\bigg\|<\eta, \quad b\in \mc{F}\cup\{\hat{h},\bar{h}\},
\end{equation}
and
\begin{equation}\label{eq:mainproof2}
\|\varphi^{(i)}(x_{\bar{h}}^{(i)})\cdot b-\varphi^{(i)}(x_b^{(i)})\|<\eta,\quad b\in\mc{F}\cup\{\hat{h}\},\; i=0,\dots,d.
\end{equation}
For $0\le i\le d$ set $C_i\coloneqq \mathrm{C}^\ast(D,\varphi^{(i)}(F^{(i)}))$ and consider the set $\mathcal{F}_i\coloneqq \{\varphi^{(i)}(x_b^{(i)}): b\in\mc{F}\}\Subset C_i$. Apply \Cref{prop:cond-exps} for each $C_i$ with the above $\eta>0$ and the finite set $\mc{F}_i$ and denote the maps we obtain from that by $\Psi_i\colon A\to C_i$. For each $0\le i\le d$ set  $d_i\coloneqq \varphi^{(i)}(x_{\bar{h}}^{(i)})\cdot \hat{h}\in D^1$,  and note that $d_ib = \varphi^{(i)}(x_{\bar{h}}^{(i)}) b$, so \ref{item:condexp3} of \Cref{prop:cond-exps} and a triangle inequality yield that
\begin{equation}\label{eq:mainproof3}
\|\Psi_i(d_ib)-d_ib\|<3\eta,\quad b\in\mc{F},\; i=0,\dots, d.
\end{equation}
Moreover, $\Psi_i(d_ib)\in B\cap C_i$ for all $0\le i\le d$ and $b\in\mc{F}$ by \ref{item:condexp1} of \Cref{prop:cond-exps}. Also, by a multiplicative domain argument, we have
\begin{equation}
\bar{h}\cdot \Psi_i(d_ib) = \Psi_i(\bar{h}d_ib) = \Psi_i(d_ib) = \Psi_i(d_ib\bar{h})=\Psi_i(d_ib)\cdot\bar{h}
\end{equation}
and also $\bar{h}d_i=d_i\bar{h}=d_i$. Applying \Cref{thm:subhom}, for each $i=0,\dots,d$ we have that 
\begin{equation*}
\dim_\mathrm{diag}(D\subset B\cap C_i) = \dim\widehat{D}\eqqcolon k\in\bb{N}.
\end{equation*} 
Setting $\tilde{\mathcal{F}}_i\coloneqq\{\Psi_i(d_ib):b\in \mc{F}\} \Subset (B\cap C_i)^1$, we apply \cite[Proposition~2.3]{LiLiaWin23} for each pair $(D\subset B\cap C_i)$ with our $\eta>0$, the finite set $\tilde{\mc{F}}_i\cup\{d_i\}$  and $\bar{h}\in D_+^1$ as a unit for $\tilde{\mc{F}}_i\cup\{d_i\}$. In this way, for each $i=0,\dots,d$, we obtain finite-dimensional $\mathrm{C}^\ast$-algebras $\bigoplus_{j=0}^k\tilde{F}^{(i,j)}$ together with masas $\bigoplus_{j=0}^kD_{\tilde{F}^{(i,j)}}\subset\bigoplus_{j=0}^k\tilde{F}^{(i,j)}$ and c.p.c.\ order zero maps $\tilde{\varphi}^{(i,j)}\colon\tilde{F}^{(i,j)}\to B\cap C_i$, each of which maps $\mc{N}_{\tilde{F}^{(i,j)}}(D_{\tilde{F}^{(i,j)}})$ inside $\mc{N}_A(D)$, and elements $y_a^{(i)}=\bigoplus_{j=0}^ky_a^{(i,j)}\in(\bigoplus_{j=0}^k\tilde{F}^{(i,j)})^1$ for $a\in\tilde{\mc{F}}_i$ and $y_{d_i}^{(i)} = \bigoplus_{j=0}^ky_{d_i}^{(i,j)}\in (\bigoplus_{j=0}^kD_{\tilde{F}^{(i,j)}})^1$ and $y_{\bar{h}}^{(i)}=\bigoplus_{j=0}^ky_{\bar{h}}^{(i,j)}\in (\bigoplus_{j=0}^kD_{\tilde{F}^{(i,j)}})^1$ such that\footnote{There is a subtlety here: a direct application of \cite[Proposition~2.3]{LiLiaWin23} as stated therein would not necessarily give a contraction $y_{d_i}^{(i)}$ \emph{in the finite-dimensional masa}. Nevertheless, an inspection of the proof of the direction of \cite[Proposition~2.3]{LiLiaWin23} that we are using here shows that we can take $y_{d_i}^{(i)}$ in $\bigoplus_{j=0}^k D_{\tilde{F}^{(i,j)}}$ without harming the generality.}
\begin{equation}\label{eq:mainproof7}
\bigg\|a-\sum_{j=0}^k\tilde{\varphi}^{(i,j)}(y_{a}^{(i,j)})\bigg\|<\eta,\quad a\in\tilde{\mc{F}}_i\cup\{d_i,\bar{h}\},\; 0\le i\le d
\end{equation}
and
\begin{equation}\label{eq:mainproof8}
\|\tilde{\varphi}^{(i,j)}(y_{\bar{h}}^{(i,j)})\cdot a - \tilde{\varphi}^{(i,j)}(y_a^{(i,j)})\|<\eta,\quad a\in\tilde{\mc{F}}_i\cup\{d_i\},\; 0\le j\le k,\; 0\le i\le d.
\end{equation}

For each $0\le i\le d$ and $0\le j\le k$ and for each $b\in\mc{F}$ set $y_b^{(i,j)}\coloneqq y_{\Psi_i(d_ib)}^{(i,j)}$, and  set $y_h^{(i,j)} \coloneqq y_{d_i}^{(i,j)}$. For $b\in\mc{F}$ we have that
\begin{align*}
\bigg\|\sum_{i=0}^d\sum_{j=0}^k\tilde{\varphi}^{(i,j)}(y_b^{(i,j)}) - b\bigg\| &\overset{(\ref{eq:mainproof1})}{<} \bigg\|\sum_{i=0}^d\sum_{j=0}^k\tilde{\varphi}^{(i,j)}(y_b^{(i,j)}) - \sum_{i=0}^d\varphi^{(i)}(x_b^{(i)})\bigg\|+\eta\\
&\overset{\;\;\;\;\;\;\;}{\le} \sum_{i=0}^d\bigg\|\sum_{j=0}^k\tilde{\varphi}^{(i,j)}(y_b^{(i,j)}) - \varphi^{(i)}(x_b^{(i)})\bigg\| +\eta\\
&\overset{\eqref{eq:mainproof7}}{<} \sum_{i=0}^d\|\Psi_i(d_ib) - \varphi^{(i)}(x_b^{(i)})\| + (d+2)\eta\\
&\overset{\eqref{eq:mainproof3}}{<} \sum_{i=0}^d\|d_ib-\varphi^{(i)}(x_b^{(i)})\|+(4d+5)\eta\\
&\overset{\eqref{eq:mainproof2}}{<}(5d+6)\eta.
\end{align*}
Also,
\begin{align*}
\bigg\|\sum_{i=0}^d\sum_{j=0}^k\tilde{\varphi}^{(i,j)}(y_h^{(i,j)})-h\bigg\|&\overset{\;\;\;\;\;\;\;}{<} \bigg\|\sum_{i=0}^d\sum_{j=0}^k\tilde{\varphi}^{(i,j)}(y_{h}^{(i,j)}) - \hat{h} \bigg\|+ \eta\\
&\overset{(\ref{eq:mainproof1})}{<}  \bigg\|\sum_{i=0}^d\sum_{j=0}^k\tilde{\varphi}^{(i,j)}(y_{h}^{(i,j)}) - \sum_{i=0}^d \varphi^{(i)}(x_{\hat{h}}^{(i)})\bigg\| +2\eta \\
&\overset{\eqref{eq:mainproof7}}{<}\bigg\|\sum_{i=0}^d(d_i-\varphi^{(i)}(x_{\hat{h}}^{(i)}))\bigg\|+(d+3)\eta\\
&\overset{\;\;\;\;\;\;\;}{\le}\sum_{i=0}^d\|\varphi^{(i)}(x_{\bar{h}}^{(i)})\cdot \hat{h} - \varphi^{(i)}(x_{\hat{h}}^{(i)})\| +(d+3)\eta\\
&\overset{\eqref{eq:mainproof2}}{<}  (2d+4)\eta.
\end{align*}
Finally, for $b\in\mc{F}$ we have that
\begin{align*}
\|\tilde{\varphi}^{(i,j)}(y_h^{(i,j)}) b - \tilde{\varphi}^{(i,j)}(y_b^{(i,j)})\| &\overset{\eqref{eq:mainproof8}}{<}  \|\tilde{\varphi}^{(i,j)}(y_h^{(i,j)}) b  - \tilde{\varphi}^{(i,j)}(y_{\bar{h}}^{(i,j)}) \Psi_i(d_ib)\| +\eta \\
&\overset{\eqref{eq:mainproof8}}{<} \|\tilde{\varphi}^{(i,j)}(y_{\bar{h}}^{(i,j)})\cdot (d_ib-\Psi_i(d_ib))\| + 2\eta\\
&\overset{\eqref{eq:mainproof3}}{<} 5\eta.
\end{align*}
These estimates together with the fact that all the maps $\tilde{\varphi}^{(i,j)}$ take values in $B$ yield
\begin{equation*}
\dim_\mathrm{diag}(D\subset B)\le (\dim\widehat{D}+1) \cdot (d+1)-1
\end{equation*}
as we wanted.
\end{proof}

\begin{remark}\label{rmk:optimal}
It is not clear at the moment whether the upper bound of \eqref{eq:main} can be improved by removing the factor $\dim^{+1}\widehat{D}$. More precisely, we have no example of a $\mathrm{C}^\ast$-diagonal $(D\subset A)$ with spectrum of covering dimension at least one (since for zero-dimensional spectra this discussion is vacuous) and an intermediate sub-$\mathrm{C}^\ast$-algebra $B$ such that the diagonal dimension of $(D\subset B)$ is greater than that of $(D\subset A)$. 

The diagonal dimension is strongly related to the dynamic asymptotic dimension of the underlying Weyl groupoid (in the sense of \cite[Definition~4.2]{Ren08}), but the precise relation remains unclear (cf.\ \cite[Remarks~6.8 and 6.9, Question~6.11]{LiLiaWin23}). Since by \cite{BrownExelFullerPittsReznikoff21, BrownExelFullerPittsReznikoff21corrigendum} an intermediate sub-$\mathrm{C}^\ast$-algebra $B$ in our case is the (possibly twisted) groupoid $\mathrm{C}^\ast$-algebra of a wide,\footnote{A subgroupoid is called \emph{wide} when it has the same unit space as the ambient groupoid.} open subgroupoid of the Weyl groupoid of $(D\subset A)$, and since dynamic asymptotic dimension seems to be monotone with respect to wide open subgroupoids \cite{Bonicke24}, obtaining a precise relation between the diagonal dimension of a pair and the dynamic asymptotic dimension of the underlying twisted groupoid could possibly shed more light on this. To the best of my knowledge, this topic will be further explored in upcoming work of C.\ B{\"o}nicke and K.\ Li.
\end{remark}

We finish this paper by establishing \Cref{cori:B} from \Cref{thmi:A}.

\begin{proof}[Proof of Corollary~B]
Assume that $(D\subset A)$ is a $\mathrm{C}^\ast$-diagonal with $A$ unital and separable, such that $\dim_\mathrm{diag}(D\subset A)<\infty$. If $D\subset B\subset A$ is an intermediate sub-$\mathrm{C}^\ast$-algebra, by \Cref{thm:main} together with \cite[Remarks~2.2~(i)]{LiLiaWin23} we have that $B$ has finite nuclear dimension. Since $1_A\in D\subset B$, $B$ is also unital and separable. By \cite{BarLi17} it also satisfies the UCT, so, whenever it is moreover simple, it is classifiable.
\end{proof}

\end{document}